\documentclass[11pt]{article}

\usepackage[margin=1in]{geometry}
\usepackage{amsmath,amssymb,amsthm,mathtools,mathrsfs}
\usepackage{enumitem}
\usepackage{hyperref}
\usepackage{cleveref}
\usepackage{bm}
\usepackage{tikz-cd}
\usepackage{array}
\usepackage{booktabs}
\usepackage{microtype}
\usepackage{xcolor}
\usepackage{authblk}

\hypersetup{colorlinks=true,linkcolor=blue,citecolor=blue,urlcolor=blue}

\newtheorem{theorem}{Theorem}[section]
\newtheorem{proposition}[theorem]{Proposition}
\newtheorem{lemma}[theorem]{Lemma}
\newtheorem{corollary}[theorem]{Corollary}
\theoremstyle{definition}
\newtheorem{definition}[theorem]{Definition}
\newtheorem{remark}[theorem]{Remark}
\newtheorem{example}[theorem]{Example}

\newcommand{\R}{\mathbb{R}}

\newcommand{\N}{\mathbb{N}}
\newcommand{\Z}{\mathbb{Z}}
\newcommand{\supp}{\operatorname{supp}}
\newcommand{\vect}{\operatorname{Vec}}

\newcommand{\ReLU}{\operatorname{ReLU}}
\newcommand{\Stack}{\operatorname{Stack}}
\newcommand{\Fix}{\operatorname{Fix}}

\newcommand{\cS}{\mathcal{S}}

\newcommand{\Ups}{\Upsilon}
\newcommand{\transpose}{\mathsf{T}}

\newcommand{\cube}{\mathcal{C}}

\newcommand{\mar}{\varrho}

\newcommand{\doi}[1]{\href{https://doi.org/#1}{DOI: #1}}

\title{Exact Loop Controllers for ReLU Realization of Homogeneous Curve Refinements\footnote{Extended arXiv version containing additional corollaries and examples}}
\author[1,2]{Boldsaikhan Bolorkhuu}
\author[1,2,3]{Tsogtgerel Gantumur}

\affil[1]{McGill University}
\affil[2]{National University of Mongolia}
\affil[3]{Institute of Mathematics and Digital Technology, Mongolian Academy of Sciences}

\date{April 28, 2026}

\begin{document}
\maketitle

\begin{abstract}
We study homogeneous refinement operators of the form
\[
(V\gamma)(t)=\sum_{j\in\mathbb Z}A_j\gamma(Mt-j),
\]
acting on compactly supported continuous piecewise linear curves
\(\gamma:\mathbb R\to\mathbb R^p\), where \(M\ge2\) and only finitely many matrices
\(A_j\in\mathbb R^{p\times p}\) are nonzero. We prove that the iterates \(V^n\gamma\)
admit exact ReLU realizations of fixed width and depth \(O(n)\).
The main new ingredient is an exact loop controller for the residual dynamics. Instead of
propagating scalar residual surrogates, the construction transports the residual orbit by a
forward-exact state on a polygonal loop. Scalar factors and digit selectors are then recovered from
this loop state by complementary CPwL readouts. The loop seam is not removed, but its remaining
ambiguity is confined to the final readout/selector stage, where it is harmless because the scalar
atom is supported away from the seam. This gives a homogeneous \(M\)-ary vector-valued extension
of the scalar binary refinable-function construction with a more geometric controller architecture.

We also record crude exponential bounds on the network weights and biases. As consequences, affine
forcing terms are treated by expanding affine iterates into finite sums of homogeneous iterates,
yielding exact fixed-width realizations with depth \(O(n^2)\), and anchored open curves reduce to
compactly supported defects with affine anchor mismatch. Finally, we describe homogeneous
polygonal generators, including dragon-type examples and a self-intersecting Hilbert-type prototype
in arbitrary dimension.

The extended version also records additional reductions and examples not included in the journal
version. These include stage-dependent forcing, finite-state
stacking reductions, and further geometric constructions such as Koch-, Gosper-, Morton-, and connector-based
Hilbert-type variants.
\end{abstract}

\tableofcontents

\section{Introduction}
\label{sec:introduction}

\subsection{Background and motivation}

Neural-network approximation theory \cite{approx,nnapprox} seeks structural explanations for why deep networks can
efficiently represent highly oscillatory, highly recursive, or highly self-similar functions.
A particularly striking result in this direction is the scalar binary theorem on refinable functions
\cite{source}: if
\[
(Vg)(x)=\sum_{j=0}^{N} c_j g(2x-j),
\]
and \(g\) is compactly supported and continuous piecewise linear, then the iterates \(V^n g\)
admit fixed-width ReLU realizations with depth growing linearly in \(n\).
The proof is built around a cascade representation, a controlled replacement of discontinuous digit
extraction by CPwL (continuous and piecewise linear) surrogates, and a recursive fixed-depth update block.

The purpose of the present paper is to develop a homogeneous \(M\)-ary vector-valued extension of
the scalar binary picture.
More precisely, we consider operators
\[
(V\gamma)(t)=\sum_{j\in\mathbb Z}A_j\gamma(Mt-j),
\]
where \(M\ge 2\) is an integer, the matrices \(A_j\in\mathbb R^{p\times p}\) vanish for all but finitely
many \(j\), and \(\gamma:\mathbb R\to\mathbb R^p\) is a compactly supported CPwL curve.
Our main result shows that the iterates \(V^n\gamma\) admit exact ReLU realizations of fixed width
and depth \(O(n)\).

The new feature is not merely the passage from binary scalar refinement to \(M\)-ary vector-valued
refinement.  At the level of realizability, such an extension can be obtained by adapting the
scalar-surrogate machinery of \cite{source}.  The point here is a different controller architecture.
The residual dynamics is naturally circular, and we encode it by an exact state evolving on a
polygonal loop.  Thus the construction no longer has to propagate approximate scalar residuals
forward through the network.  Instead, the residual position is transported exactly, and the scalar
factors and digit selectors needed by the cascade are recovered from this loop state by complementary
CPwL readouts.

This does not remove the seam issue altogether.  The loop identifies the two endpoints of the
unit interval, so no single continuous readout can serve as a globally unambiguous coordinate.
The construction instead arranges that the remaining ambiguity appears only at the final scalar
readout and selector step, where it is harmless because the scalar atom is supported away from the
seam.  Thus the gain over scalar surrogate iteration is that the forward residual stream is exact,
while the unavoidable seam ambiguity is handled only at the end of the construction.

Our original motivation came from a viewpoint complementary to that of \cite{fractals}.
There, the emphasis is on constructing neural-network functions that act as nearly exact
indicators for fractal sets. The present work starts from the dual question of \emph{fractal generation}:
can one construct recursive fractal \emph{curves} themselves, as vector-valued neural-network outputs
obtained from successive refinement steps?
This is dual in much the same way that a curve can be described either by a parametrization
or by an implicit equation.
From this point of view, passing from scalar functions to vector-valued curves is not a cosmetic
extension but a practical one: it brings recursive curve constructions, including curve-based fractal
and space-filling examples, directly into the neural-network realization framework.

The broader project that led to this paper also includes affine refinement operators and
stage-dependent forcing rules.
In the present paper, however, we deliberately center the exact \(O(n)\)-depth theory on the
homogeneous case, where the loop-controller mechanism is cleanest and most transparent.
Affine and stage-dependent situations are still discussed, but only through black-box reductions
from the homogeneous theorem, with weaker complexity bounds.

\subsection{Main results}

We now state the main realization theorem, then summarize the principal reductions and applications. 
The detailed proofs are given in Section~\ref{sec:homogeneous-extension}. 
As standardized by \cite{approx,nnapprox}, we write
\(\Ups_{W,L}(\ReLU;d,N)\)
for the set of outputs of fully connected ReLU networks with width \(W\), depth \(L\), input dimension \(d\), and output dimension \(N\).

\begin{theorem}
\label{thm:intro-homogeneous}
Fix an integer \(M\ge 2\).
Let \(\gamma:\R\to\R^p\) be a compactly supported CPwL curve with
\(\supp\gamma\subset[0,L]\), and assume that \(V\) preserves the support window \([0,L]\).
If \(\gamma\) has at most \(m\) breakpoints, then there exist constants \(C_0,C_1>0\), depending only on \(M\), \(p\), \(L\), and \(m\), such that
\[
V^n\gamma\in \Ups_{C_0,C_1 n}(\ReLU;1,p), \qquad n\ge 1.
\]
In particular, the width remains bounded independently of \(n\) and the depth is linear in \(n\).
\end{theorem}

At a secondary level of bookkeeping, the proof also yields crude exponential bounds on the magnitudes of the network weights and biases.

Beyond \Cref{thm:intro-homogeneous}, the extended version records several reductions and
applications.  Affine and stage-dependent forcing rules are treated by a finite-sum reduction from
the homogeneous theorem: their iterates are expanded into homogeneous pieces, and
\Cref{thm:intro-homogeneous} is applied term by term.  Under the corresponding CPwL
finite-complexity assumptions on the forcing terms, this gives exact fixed-width ReLU realizations
with quadratic depth \(O(n^2)\).

Open-curve constructions are handled by subtracting an anchor profile.  The resulting compactly
supported defect satisfies an affine recursion whose forcing term is the anchor mismatch.  Thus,
except in the special case where the anchor mismatch vanishes, anchored open curves naturally fall
under the affine \(O(n^2)\)-depth corollary rather than the homogeneous \(O(n)\)-depth theorem.

The extended version also records finite-state stacking reductions and additional geometric
constructions.  Recursive polygonal rules are translated into homogeneous, affine, or
stage-dependent refinement form, giving examples such as dragon-, Koch-, Gosper-, Morton-, and
Hilbert-type generators.  The forcing-free closed or endpoint-compatible homogeneous cases fall
under the linear-depth theorem; the anchored, affine, or stage-dependent variants fall under the
quadratic-depth finite-sum reductions.

\subsection{Structure of the paper}

The paper is organized as follows.
Section~\ref{sec:preliminaries} fixes notation and records the basic identities behind the cascade formalism.
Section~\ref{sec:homogeneous-extension} contains the core contribution of the paper: the exact loop controller, the scalar and matrix readouts, the recursive realization on the unit interval, and the passage from atomic curves to arbitrary compactly supported CPwL curves.
Section~\ref{sec:corollaries-applications} collects the main corollaries and applications, including homogeneous finite-state systems, anchored constructions, black-box affine and stage-dependent reductions, and geometric examples.
We end with a brief conclusion.

\section{Preliminaries and notation}
\label{sec:preliminaries}

This section fixes notation and records the basic identities behind the cascade formalism.
The underlying logic is the same as in the scalar binary treatment of \cite{source}, but is recast
here for the homogeneous \(M\)-ary vector-valued setting.
We begin with the \(M\)-ary digit and residual maps, then introduce the vectorization and block
transition matrices, and finally record the one-step and iterated cascade identities together with
the special-hat decomposition and translation covariance that will be used in the proof of the main theorem.

\subsection{Network classes and CPwL functions}

Following \cite{source,approx,nnapprox},
for integers \(W,L,d,N\ge 1\), we write
\[
\Ups_{W,L}(\ReLU;d,N)
\]
for the set of outputs of fully connected ReLU networks with width \(W\), depth \(L\), input
dimension \(d\), and output dimension \(N\).
Only the one-dimensional input case \(d=1\) is used in this paper, but the output dimension varies.

We write \emph{CPwL} for \emph{continuous piecewise linear}.
Thus a function \(f:\R\to\R^N\) is CPwL if each component is continuous and piecewise affine with
finitely many breakpoints on every compact interval.
For compactly supported CPwL maps, this is equivalent to the existence of a finite partition of the
support into intervals on each of which \(f\) is affine.

\begin{remark}
\label{rem:cpwl-basic}
We use repeatedly the standard facts that:
\begin{enumerate}[label=(\roman*),leftmargin=2em]
    \item every scalar CPwL function with finitely many breakpoints is an output of a
    one-hidden-layer ReLU network of width proportional to the number of affine pieces;
    \item compositions of CPwL functions are again CPwL;
    \item translating the input of a ReLU network by an affine map changes neither the asymptotic
    width nor the depth;
    \item finite sums of fixed-width, depth-\(O(n)\) networks can be realized by a network of
    width enlarged by a constant factor and depth still \(O(n)\).
\end{enumerate}
When coefficient bounds are tracked, we only use crude exponential estimates in \(n\).
In particular, the complexity measure in this paper is primarily architectural, namely width and
depth, while bounds on weights and biases are recorded separately and only at a coarse level.
\end{remark}

\subsection{\(M\)-ary digit maps and residuals}
\label{ss:digits-residuals}

Fix an integer \(M\ge 2\).
For \(x\in[0,1)\), we define the digit map
\begin{equation}
\label{eq:digit-map}
Q(x):=\lfloor Mx\rfloor\in\{0,1,\dots,M-1\},
\end{equation}
and the associated residual map
\begin{equation}
\label{eq:residual-map}
R(x):=Mx-Q(x)=Mx-\lfloor Mx\rfloor.
\end{equation}
Thus \(Q\) selects the first digit in the preferred base-\(M\) expansion, and
\(R(x)\in[0,1)\) for \(x\in[0,1)\).
For later convenience, we extend these definitions to the endpoint \(x=1\) by continuity, setting
\(R(1)=1\) and \(Q(1)=M-1\).
Accordingly, we regard \(R\) and \(Q\) as maps on \([0,1]\), with the usual formulas on \([0,1)\)
and this endpoint convention at \(x=1\).

On \([0,1)\), the formula \(R(x)=Mx-\lfloor Mx\rfloor\) is the standard mod-\(1\) residual map, and
hence induces the usual circle dynamics on \(\R/\Z\). Our convention \(R(1)=1\) is simply the
closed-interval representative of the seam point \(0\sim 1\). This point of view is not merely
formal: it underlies the exact loop-based residual controller introduced later in
Section~\ref{subsec:section3-loop-controller}.
Figure~\ref{fig:digit-res} illustrates the first two digit and residual maps in the ternary case, including the endpoint convention at the breakpoints.

\begin{figure}[ht]
\centering
\includegraphics[width=.7\textwidth]{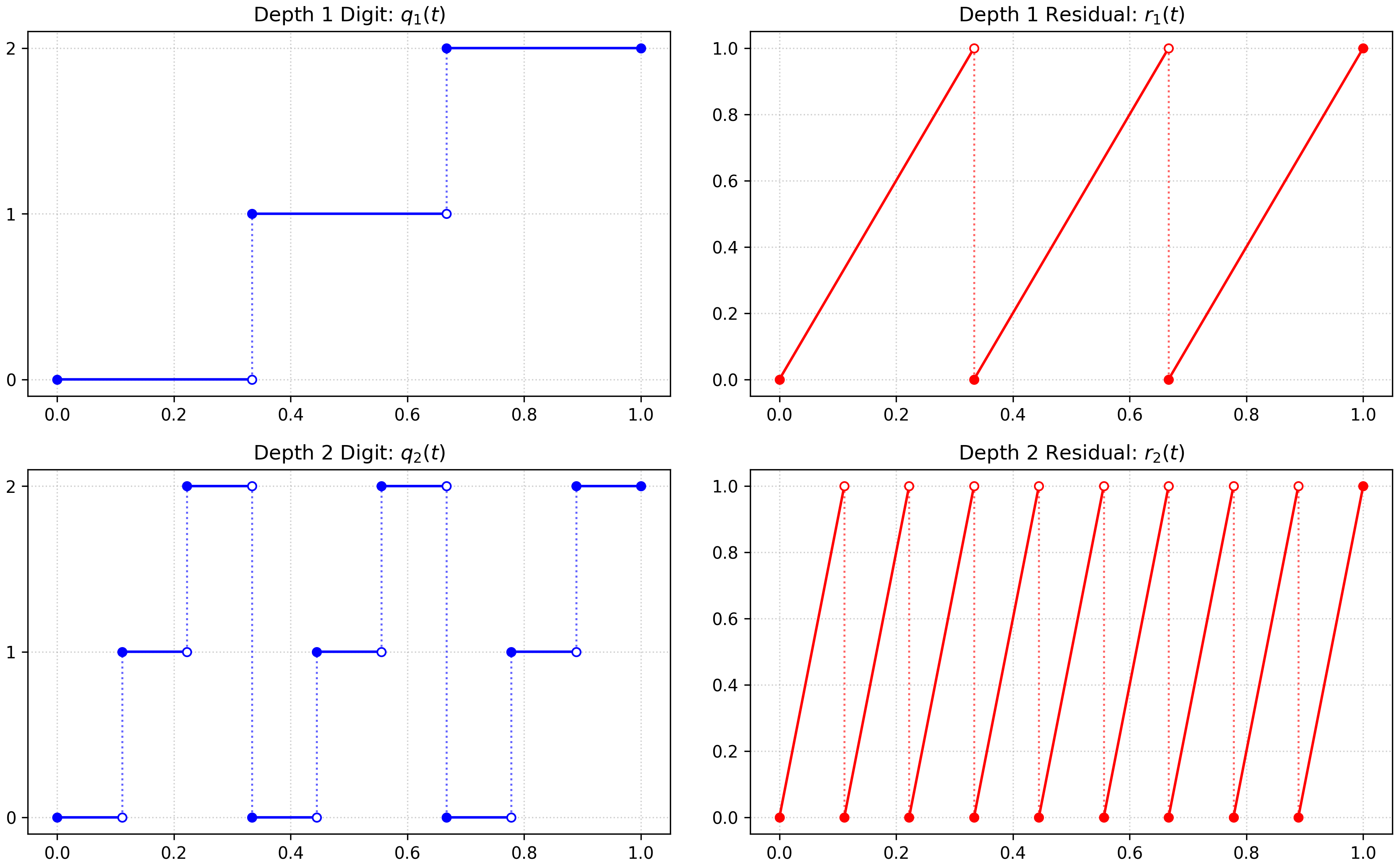}
\caption{First and second ternary digit and residual maps.}
\label{fig:digit-res}
\end{figure}

The successive digits and residuals are defined for \(j\ge 1\) by
\begin{equation}
\label{eq:digits-residuals}
q_j(x):=Q(R^{j-1}(x)),
\qquad
R^j(x):=R\circ R^{j-1}(x),
\end{equation}
with \(R^0(x)=x\).
Since \(R([0,1])\subset[0,1]\), these iterates are well defined for every \(x\in[0,1]\).
Then every \(x\in[0,1]\) has the preferred base-\(M\) expansion
\begin{equation}
\label{eq:M-ary-expansion}
x=\sum_{j=1}^{\infty} q_j(x) M^{-j},
\end{equation}
with the usual ambiguity at \(M\)-ary rationals resolved by the above convention.

\begin{remark}
\label{rem:affine-form-of-residual-iterates}
For each \(n\ge 1\), and on each \(M\)-ary cell
\[
\Big[\frac{\ell}{M^n},\frac{\ell+1}{M^n}\Big),
\qquad \ell=0,1,\dots,M^n-1,
\]
the iterate \(R^n\) is affine:
\[
R^n(x)=M^n x-\ell.
\]
At the endpoint \(x=1\), the same formula holds with \(\ell=M^n-1\), giving \(R^n(1)=1\).
\end{remark}

\subsection{The homogeneous operator and support convention}

Fix matrices \(A_j\in\R^{p\times p}\), \(j\in\mathbb Z\), and assume that \(A_j=0\) for all but finitely
many \(j\).
Define the homogeneous refinement operator
\begin{equation}
\label{eq:homogeneous-operator}
(V\gamma)(t)=\sum_{j\in\mathbb Z} A_j\gamma(Mt-j).
\end{equation}
Throughout most of the paper we work with compactly supported curves \(\gamma:\R\to\R^p\).
Fix an integer \(L\ge 1\) such that
\begin{equation}
\label{eq:support-convention}
\supp\gamma\subset [0,L].
\end{equation}
We assume that the operator data are chosen so that \(V\) preserves this support window, namely
\begin{equation}
\label{eq:support-preserve}
\supp(V\gamma)\subset [0,L] ,
\end{equation}
whenever \eqref{eq:support-convention} holds.

\subsection{Vectorization}

For \(k=1,\dots,L\), define the \(k\)th block of \(\gamma\) by
\begin{equation}
\label{eq:block-definition}
\gamma_k(x):=\gamma(x+k-1), \qquad x\in[0,1].
\end{equation}
The vectorization of \(\gamma\) is then the function
\begin{equation}
\label{eq:vectorization}
G(x):=\vect(\gamma)(x):=\bigl(\gamma_1(x),\gamma_2(x),\dots,\gamma_L(x)\bigr)^\transpose\in\R^{pL},
\qquad x\in[0,1].
\end{equation}
Thus \(G\) records all unit-interval pieces of \(\gamma\) simultaneously, reparameterized to \([0,1]\).
Then for \(n\ge0\), the $n$-th iterate $V^n\gamma$ is encoded by
\begin{equation}
\label{eq:Gn-definition}
G^n(x)=\vect(V^n\gamma)(x)\in\R^{pL}.
\end{equation}
Note that the special case $n=0$ reduces to $G^0=G$.

\begin{remark}
\label{rem:vectorization-geometry}
The vectorization operator is the correct analogue, in the present setting, of the formalism used in \cite{source}.
It packages all translated pieces of the curve into one state vector.
The refinement step then becomes multiplication by a block matrix selected by the current
\(M\)-ary digit.
\end{remark}

\subsection{Block transition matrices}
\label{ss:block-transition}

For each digit value \(q\in\{0,1,\dots,M-1\}\), define the block transition matrix
\(T_q\in\R^{pL\times pL}\)
by declaring its \((k,\ell)\)-block, of size \(p\times p\), to be
\begin{equation}
\label{eq:Tq-block}
(T_q)_{k\ell}:=A_{\,q+M(k-1)-(\ell-1)},
\qquad 1\le k,\ell\le L.
\end{equation}
Here we understand \(A_j=0\) whenever the corresponding shift is absent.

The point of this definition is that on the cell where the first digit is \(q\), each argument
\(Mt-j\) differs from the common residual variable \(R(t)\) by an integer shift; the matrix
\(T_q\) records which translated block of \(\gamma\) contributes to which translated block of
\(V\gamma\).

\subsection{The cascade identity}
\label{ss:cascade-identity}

The next proposition is the vector-valued \(M\)-ary version of the standard one-step cascade
relation.
It is the starting point for all later recursive constructions.

\begin{proposition}[One-step cascade identity]
\label{prop:one-step-cascade}
For every \(x\in[0,1]\), we have
\begin{equation}
\label{eq:one-step-cascade}
G^1(x)=T_{q_1(x)}G(R(x)).
\end{equation}
\end{proposition}

\begin{proof}
Fix \(x\in[0,1]\) and set \(q:=q_1(x)=Q(x)\).
For each block index \(k\in\{1,\dots,L\}\),
\[
(G^1)_k(x)=(V\gamma)(x+k-1)=\sum_{j\in\mathbb Z} A_j\gamma(Mx+M(k-1)-j).
\]
Since \(Mx=R(x)+q\), this becomes
\[
(G^1)_k(x)=\sum_{j\in\mathbb Z} A_j\gamma\bigl(R(x)+q+M(k-1)-j\bigr).
\]
For a fixed \(j\), let
\(\ell:=q+M(k-1)-j+1\).
If \(\ell\in\{1,\dots,L\}\), then by definition of the blocks,
\[
\gamma\bigl(R(x)+q+M(k-1)-j\bigr)=\gamma_\ell(R(x))=G_\ell(R(x)).
\]
If \(\ell\notin\{1,\dots,L\}\), then the corresponding argument lies outside the support window and
the contribution vanishes.
Therefore we infer
\[
(G^1)_k(x)=\sum_{\ell=1}^{L} A_{\,q+M(k-1)-(\ell-1)}\,G_\ell(R(x))
=\sum_{\ell=1}^{L} (T_q)_{k\ell} G_\ell(R(x)),
\]
which is exactly the \(k\)th block of \(T_qG(R(x))\).
Since \(k\) was arbitrary, \eqref{eq:one-step-cascade} follows.
\end{proof}

\begin{corollary}[Iterated cascade identity]
\label{cor:iterated-cascade}
For every \(n\ge 1\) and every \(x\in[0,1]\),
\begin{equation}
\label{eq:iterated-cascade}
G^n(x)=T_{q_1(x)}T_{q_2(x)}\cdots T_{q_n(x)}\,G(R^n(x)).
\end{equation}
\end{corollary}

\begin{proof}
We argue by induction on \(n\).
For \(n=1\), the statement is exactly \Cref{prop:one-step-cascade}.
Assume the formula holds for some \(n\ge 1\). Then we have
\[
G^{n+1}(x)=\vect(V^{n+1}\gamma)(x)=\vect\bigl(V(V^n\gamma)\bigr)(x).
\]
Applying \Cref{prop:one-step-cascade} to the curve \(V^n\gamma\), we obtain
\[
G^{n+1}(x)=T_{q_1(x)}\,G^n(R(x)).
\]
Now apply the induction hypothesis at the point \(R(x)\), to get
\[
G^n(R(x))
=
T_{q_1(R(x))}T_{q_2(R(x))}\cdots T_{q_n(R(x))}\,G(R^n(R(x))).
\]
By definition of the digits, we infer
\[
q_j(R(x))=Q(R^{j-1}(R(x)))=Q(R^j(x))=q_{j+1}(x),
\qquad j=1,\dots,n.
\]
Also, we have \(R^n(R(x))=R^{n+1}(x)\), and thus
\[
G^n(R(x))
=
T_{q_2(x)}T_{q_3(x)}\cdots T_{q_{n+1}(x)}\,G(R^{n+1}(x)).
\]
Substituting this into the previous formula gives
\[
G^{n+1}(x)
=
T_{q_1(x)}T_{q_2(x)}\cdots T_{q_{n+1}(x)}\,G(R^{n+1}(x)),
\]
which is exactly \eqref{eq:iterated-cascade} with \(n+1\) in place of \(n\).
\end{proof}

\subsection{Special hats and atomic curves}

Fix once and for all \(0<\mar<\frac12\).
The specific value of \(\mar\) is not important.
Its role is simply to keep the support of the distinguished hat away from the endpoints \(0\) and
\(1\).

\begin{definition}[Special-hat]
\label{def:special-hat}
A nonnegative CPwL function \(h:\mathbb R\to\mathbb R\) is called a \emph{special-hat} if
\[
\supp h \subset[\mar,1-\mar] .
\]
\end{definition}

For \(\mu\in\{1,\dots,p\}\), let \(e_\mu\in\R^p\) be the $\mu$th standard coordinate vector.
An \emph{atomic curve} is a curve of the form
\begin{equation}
\label{eq:special-basis-curve}
\gamma(t)=h(t)e_\mu,
\end{equation}
where \(h\) is a special hat.
For such curves the vectorization simplifies drastically: on \([0,1]\) one has
\[
G(x)=h(x)u_\mu,
\]
where \(u_\mu\in\R^{pL}\) is the coordinate vector corresponding to the \(\mu\)th coordinate of the
first block.
This is the point at which the one-dimensional residual problem and the matrix-valued cascade
separate.

\subsection{Atomic-curve decomposition}
\label{ss:special-hat-decomp}

To pass from atomic curves to arbitrary compactly supported CPwL data, we use the standard
nodal decomposition into translated hats.

\begin{lemma}[Finite atomic-curve decomposition]
\label{lem:finite-hat-decomposition}
Every compactly supported CPwL map \(F:\mathbb R\to\mathbb R^p\) can be written as a finite sum
\[
F(t)=\sum_{\nu=1}^{N_*} a_\nu\, h_\nu(t-\delta_\nu)\,v_\nu,
\]
where \(a_\nu\in\mathbb R\), \(\delta_\nu\in\mathbb R\), \(v_\nu\in\mathbb R^p\), and each \(h_\nu\) is a
special-hat with at most three breakpoints.
In particular, every compactly supported CPwL curve is a finite linear combination of translated atomic curves.
\end{lemma}

\begin{proof}
Choose a finite breakpoint set containing all breakpoints of all coordinate functions of \(F\), and
refine it further if necessary so that every nodal hat on that grid has support length strictly less
than \(1-2\mar\).
Each scalar coordinate of \(F\) is then a finite linear combination of nodal hats on that refined grid.
For each such nodal hat \(\psi\), the support-length condition allows one to choose a translation
\(\delta\) so that the shifted function \(h(t):=\psi(t+\delta)\) satisfies
\(\supp(h)\subset[\mar,1-\mar]\).
Then \(h\) is a special hat with at most three breakpoints, and
\(\psi(t)=h(t-\delta)\).
Collecting these finitely many terms over all coordinates gives the stated representation.
\end{proof}

\subsection{Translation covariance}
\label{ss:translation-covariance}

A small but extremely useful device is the translation covariance of the \emph{homogeneous}
refinement operator. 
It allows termwise application of \(V^n\) to the decomposition from the preceding subsection.

\begin{lemma}[Translation covariance for \(V\)]
\label{lem:translation-covariance}
Let \(\delta\in\R\) and define
\[
\widetilde{\gamma}(t):=\gamma(t-\delta).
\]
Then for every \(n\ge 1\),
\begin{equation}
\label{eq:translation-covariance}
V^n\widetilde{\gamma}(t)=\bigl(V^n\gamma\bigr)(t-M^{-n}\delta).
\end{equation}
\end{lemma}

\begin{proof}
For one step,
\[
(V\widetilde{\gamma})(t)
=
\sum_{j\in\mathbb Z}A_j\widetilde{\gamma}(Mt-j)
=
\sum_{j\in\mathbb Z}A_j\gamma(Mt-j-\delta)
=
(V\gamma)(t-M^{-1}\delta).
\]
Iterating proves \eqref{eq:translation-covariance}.
\end{proof}

\section{Exact loop controller and the homogeneous theorem}
\label{sec:homogeneous-extension}

\subsection{Overview}

In this section we prove \Cref{thm:intro-homogeneous}.
Throughout, we work with the ordinary \(M\)-ary digit and residual maps, the vectorization over the
\(L\) unit intervals \([k-1,k]\), and the corresponding block transition matrices
\(T_q\), \(q\in\{0,1,\dots,M-1\}\).
Accordingly, the vectorized state lives in \(\R^{pL}\).

The overall strategy remains close to that of \cite{source}: one isolates the scalar factor
\(h(R^n(x))\), combines it with the block transition matrices through a recursive cascade, and then
passes from atomic curves to general compactly supported CPwL curves.
The main difference lies in the treatment of the residual dynamics.
Rather than iterating scalar surrogate residuals, we encode the residual orbit
\(R(x),R^2(x),\dots\)
into an exact forward state sequence on a loop.
More precisely, we choose a polygonal loop \(\Gamma\subset\R^2\), a parametrization
\(E:[0,1]\to\Gamma\), and a CPwL update map \(F:\R^2\to\R^2\) such that, with
\[
z_0(x)=E(x),\qquad z_{n+1}(x)=F(z_n(x)),
\]
one has
\[
z_n(x)=E(R^n(x))
\]
for every \(n\ge 0\).
Thus the forward residual controller is exact. The only remaining ambiguity comes from the seam
of the loop, and this is handled at the readout stage by complementary scalar readouts from the
loop state. Away from complementary seam neighborhoods, these readouts agree with the identity,
while the remaining ambiguity is absorbed by the fact that the special-hat \(h\) is supported away
from the endpoints of \([0,1]\).

This loop-based formulation is geometrically natural: the residual map is most naturally viewed
as a circular dynamics, and in the homogeneous setting this viewpoint yields an exact controller for
the forward residual stream while leaving the matrix recursion and the final gluing argument in
essentially the same form as in the scalar binary theory. It does not produce a globally unambiguous
selector on the closed loop; rather, it replaces forward surrogate iteration by exact loop transport
and confines the seam ambiguity to the terminal readout/selector stage.

The proof is carried out in five steps.

\begin{enumerate}[leftmargin=2em]
    \item In \S\ref{subsec:section3-loop-controller}, we construct the exact loop residual
    controller.

    \item In \S\ref{subsec:section3-scalar-readouts}, we introduce two scalar readouts from the loop
    state and use them to obtain an exact min-representation for \(h(R^n(\cdot))\).

    \item In \S\ref{subsec:section3-selectors}, we construct selector matrices from the loop state.
    These are CPwL matrix fields on the loop state space that recover the exact block transition
    matrices away from a small seam region.

    \item In \S\ref{subsec:section3-unit-interval}, we combine the scalar factor and the
    selector matrices by means of the product gadget to obtain a fixed-depth recursive realization
    block on \([0,1]\).

    \item In \S\ref{subsec:section3-global}, we pass from atomic curves to arbitrary
    compactly supported CPwL curves by atomic-curve decomposition, translation covariance, and
    finite summation.
\end{enumerate}

\subsection{Exact loop controller}
\label{subsec:section3-loop-controller}

We now construct an exact forward controller for the residual dynamics. The point is to encode the
residual orbit
\(x,R(x),R^2(x),\dots\)
into a forward state sequence on a loop, so that the controller state at stage \(n\) determines
\(R^n(x)\) exactly. This removes the need to iterate scalar surrogate residuals in the forward
dynamics. 

We begin by fixing a concrete loop model. Let
\[
a_0:=(0,0),\qquad a_1:=(1,1),\qquad a_2:=(1,0),
\]
and let \(\Delta\subset\R^2\) be the closed triangle with vertices \(a_0,a_1,a_2\). Denote by
\(\Gamma:=\partial \Delta\)
its boundary loop. We parametrize \(\Gamma\) by the piecewise affine map \(E:[0,1]\to\Gamma\)
defined by
\begin{equation}
\label{eq:triangle-parametrization}
E(t)=
\begin{cases}
(3t,3t), & 0\le t\le \frac13,\\[2mm]
(1,2-3t), & \frac13\le t\le \frac23,\\[2mm]
(3-3t,0), & \frac23\le t\le 1.
\end{cases}
\end{equation}
Then \(E\) is continuous, piecewise affine, satisfies
\(E(0)=E(1)=a_0\),
and is injective on \([0,1)\). We refer to \(a_0\) as the {\em seam point}.

Recall from \Cref{ss:digits-residuals} that the residual map is \(R(t)=Mt-\lfloor Mt\rfloor\) for \(t\in[0,1)\), with $R(1)=1$. 
We first define a loop map on \(\Gamma\) by transporting \(R\) through the parametrization \(E\).

\begin{lemma}[Loop controller map]
\label{lem:loop-controller-map}
There exists a CPwL map \(F:\R^2\to\R^2\) such that
\begin{equation}
\label{eq:F-on-loop}
F(E(t))=E(R(t)),\qquad t\in[0,1].
\end{equation}
Moreover, \(F\) can be chosen so that
\(F\in \Ups_{C_M,2}(\ReLU;2,2)\)
for a constant \(C_M\) depending only on \(M\), and the weights and biases of such a realization
are bounded in absolute value by a constant depending only on \(M\).
\end{lemma}

\begin{proof}
Define first a map \(F_\Gamma:\Gamma\to\Gamma\) by
\(F_\Gamma(E(t))=E(R(t))\) for \(t\in[0,1)\).
This is well defined. Indeed, the only nontrivial identification in the parametrization \(E\) is
\(E(0)=E(1)\), and at that point the right-hand side is consistent because
\[
E(R(0))=E(0)=a_0=E(1)=E(R(1)).
\]

Next, \(F_\Gamma\) is continuous and piecewise affine on \(\Gamma\).
Indeed, \(R\) is affine on each interval
\[
\Big[\frac{k}{M},\frac{k+1}{M}\Big),\qquad k=0,1,\dots,M-1,
\]
and \(E\) is affine on each of the three subintervals
\([0,\frac13]\), \([\frac13,\frac23]\), and \([\frac23,1]\).
Therefore \(E\circ R\) is affine on each member of a finite partition of \([0,1)\) obtained by
refining these intervals.

It remains to check compatibility at the partition points. If \(t_0\) is a breakpoint coming only
from \(E\), then continuity is immediate because \(R\) is continuous at \(t_0\) and \(E\) is continuous.
If \(t_0=\frac{k}{M}\) for some \(1\le k\le M-1\), then
\[
\lim_{t\to t_0^-}E(R(t))=E(1)=a_0=E(0)=E(R(t_0)).
\]
Finally, at the endpoint \(t=1\), the value is already consistent with the seam because
\[
E(R(1))=E(1)=a_0=E(0).
\]
Thus the affine pieces match continuously at every breakpoint. Transporting the partition through
\(E\), we obtain a finite partition of \(\Gamma\) into line segments on each of which \(F_\Gamma\) is affine.

We now extend \(F_\Gamma\) from \(\Gamma\) to a CPwL map on all of \(\mathbb R^2\).
Choose a point \(c\in\operatorname{int}\Delta\). Refine the boundary loop \(\Gamma\)
by all breakpoints of \(F_\Gamma\), and write the resulting boundary vertices as
\(b_i=E(t_i)\), ordered cyclically. Triangulate \(\Delta\) by the fan triangles
\([c,b_i,b_{i+1}]\).
Choose any value \(F(c)\in\mathbb R^2\), and extend \(F\) affinely on each fan
triangle. On each refined boundary edge \([b_i,b_{i+1}]\), this extension agrees with
\(F_\Gamma\), because both maps are affine there and agree at the endpoints. Finally,
extend outside \(\Delta\) by any fixed finite CPwL extension, for instance by adding
finitely many exterior triangles around \(\Delta\). 
This gives a global CPwL map \(F:\mathbb R^2\to\mathbb R^2\) satisfying
\eqref{eq:F-on-loop}.

The map just constructed has a number of affine pieces depending only on \(M\).
By the ReLU--finite-element representation theorem of \cite{relu-fem}, 
two-dimensional affine finite-element functions
admit exact ReLU realizations with two hidden layers. Applying this componentwise to
the two scalar components of \(F\), we obtain
\(F\in \Ups_{C_M,2}(\ReLU;2,2)\),
with \(C_M\) depending only on the finite fan refinement, hence only on \(M\). Since
the fan geometry is fixed once \(M\) is fixed, the weights and biases may also be
chosen with bounds depending only on \(M\).
\end{proof}

Only the restriction of \(F\) to \(\Gamma\) is used below. Indeed, \(z_0(x)=E(x)\in\Gamma\), and
\eqref{eq:F-on-loop} implies \(F(\Gamma)\subset\Gamma\); hence all iterates \(z_m(x)\)
remain on \(\Gamma\). Thus any global CPwL extension of \(F_\Gamma\) with the stated finite
complexity would serve equally well.

\begin{figure}[ht]
\centering
\includegraphics[width=.5\textwidth]{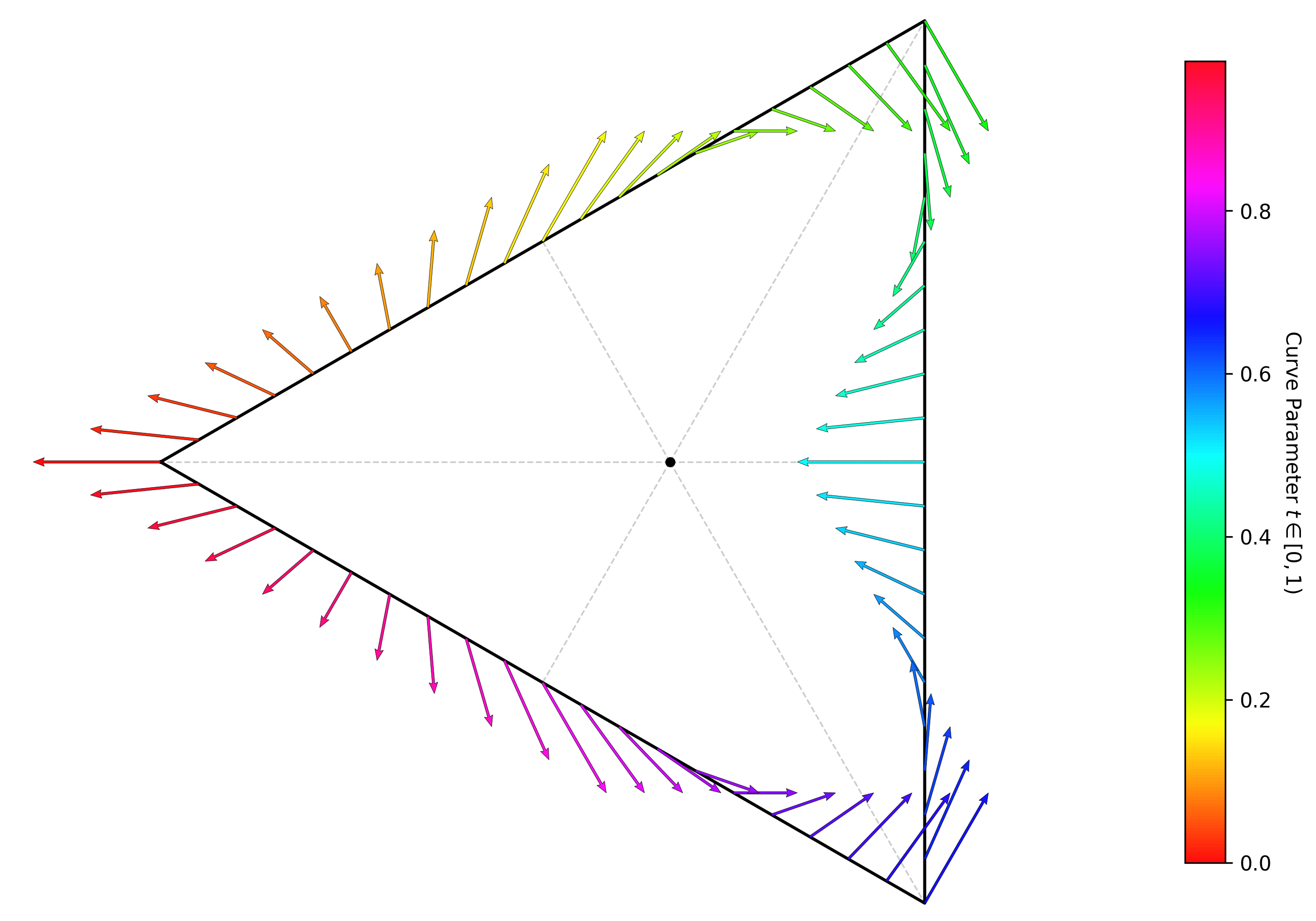}
\caption{The exact loop-controller boundary field for \(M=2\).
Arrows are based at points \(E(t)\in\Gamma\) and depict the vector value
\(F_\Gamma(E(t))=E(R(t))\). Colors encode \(t\in[0,1)\). The drawing uses an
equilateral triangle rather than the right triangle \eqref{eq:triangle-parametrization} for visual
symmetry.}
\label{fig:loop-vf}
\end{figure}

We now define the forward loop state recursively by
\begin{equation}
\label{eq:loop-state-recursion}
z_0(x):=E(x),\qquad z_{n+1}(x):=F(z_{n}(x)),\qquad n\ge0,\ x\in[0,1].
\end{equation}

\begin{proposition}[Exact forward loop controller]
\label{prop:exact-loop-controller}
For every \(n\ge 0\) and every \(x\in[0,1]\), we have
\begin{equation}
\label{eq:exact-loop-controller}
z_n(x)=E(R^n(x)).
\end{equation}
In particular, the loop state \(z_n(x)\) encodes \(R^n(x)\) exactly.
\end{proposition}

\begin{proof}
We argue by induction on \(n\). For \(n=0\), the claim is exactly the definition of \(z_0\).
Assume that \eqref{eq:exact-loop-controller} holds for some \(n\ge 0\). Then, using
\Cref{lem:loop-controller-map}, we have
\[
z_{n+1}(x)
=
F(z_n(x))
=
F(E(R^n(x)))
=
E(R(R^n(x)))
=
E(R^{n+1}(x)).
\]
This closes the induction.
\end{proof}

\subsection{Special-hats and scalar readouts}
\label{subsec:section3-scalar-readouts}

If \(\gamma=h e_\mu\) is an atomic curve, then on \([0,1]\) its vectorization has the form
\[
G(x)=\vect(\gamma)(x)=h(x)u_\mu,
\]
where \(u_\mu\in\R^{pL}\) is the standard basis vector corresponding to the \(\mu\)-th coordinate
in the first block. Hence, the cascade identity from \S\ref{ss:cascade-identity} yields
\begin{equation}
\label{eq:special-basis-cascade-loop}
G^n(x)=T_{q_1(x)}\cdots T_{q_n(x)}\,h(R^n(x))\,u_\mu,
\qquad x\in[0,1].
\end{equation}
Thus the genuinely one-dimensional factor is \(h(R^n(x))\).

By \Cref{prop:exact-loop-controller}, the stage-\(n\) loop state is
\(z_n(x)=E(R^n(x))\).
We now recover the scalar factor \(h(R^n(x))\) from this exact loop state by means of two terminal
scalar readouts.
Fix a number
\(0<\varepsilon<1\).
The precise upper bound on \(\varepsilon\) needed for the later results will be stated there.

Define two scalar functions \(r^-,r^+:[0,1]\to[0,1]\) by
\begin{equation}
\label{eq:r-minus-def}
r^-(t):=
\begin{cases}
t, & 0\le t\le 1-\varepsilon,\\[1mm]
\displaystyle \frac{1-\varepsilon}{\varepsilon}(1-t),
& 1-\varepsilon\le t\le 1,
\end{cases}
\end{equation}
and
\begin{equation}
\label{eq:r-plus-def}
r^+(t):=
\begin{cases}
\displaystyle 1-\frac{1-\varepsilon}{\varepsilon}\,t,
& 0\le t\le \varepsilon,\\[2mm]
t, & \varepsilon\le t\le 1.
\end{cases}
\end{equation}
Thus \(r^-\) is exact on \([0,1-\varepsilon]\) and folds the right seam interval
\([1-\varepsilon,1]\) linearly onto \([0,1-\varepsilon]\), while \(r^+\) is exact on
\([\varepsilon,1]\) and folds the left seam interval \([0,\varepsilon]\) linearly onto
\([\varepsilon,1]\).

Since
\(r^-(0)=0=r^-(1)\)
and
\(r^+(0)=1=r^+(1)\),
both functions are compatible with the seam identification \(0\sim 1\). Therefore they induce
well-defined continuous piecewise affine functions on the loop \(\Gamma\), which we denote by
\(\rho^\pm:\Gamma\to[0,1]\),
via
\begin{equation}
\label{eq:rho-readouts-on-loop}
\rho^\pm(E(t)):=r^\pm(t),\qquad t\in[0,1].
\end{equation}
As in the proof of \Cref{lem:loop-controller-map}, we extend \(\rho^\pm\) arbitrarily but fixedly to CPwL maps
\[
\rho^\pm:\R^2\to[0,1].
\]
We continue to denote these extensions by the same symbols.
Figure~\ref{fig:readouts} shows how the two complementary readouts recover the terminal
scalar factor from the exact loop state while avoiding the seam ambiguity.

\begin{figure}[ht]
\centering
\includegraphics[width=.7\textwidth]{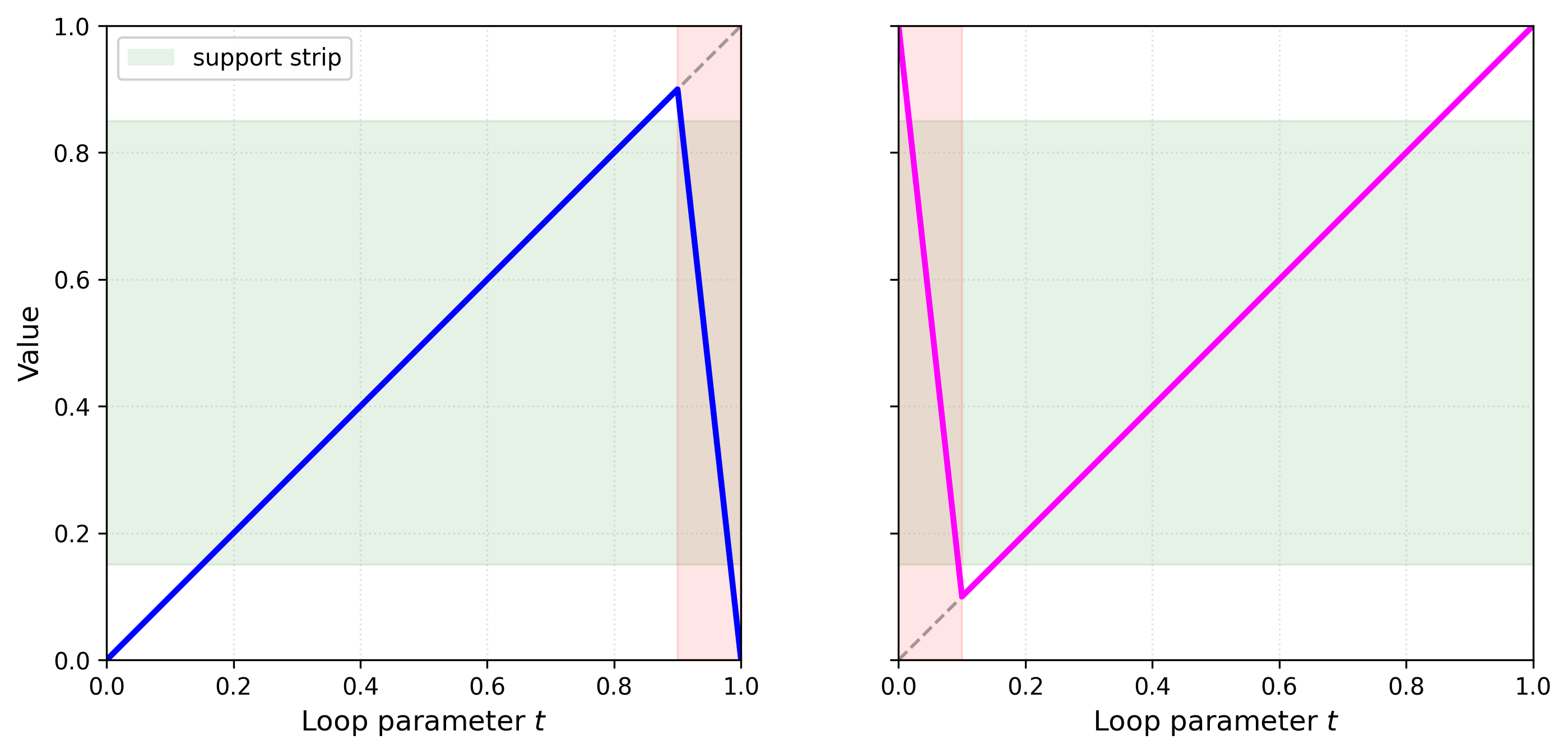}
\caption{Complementary scalar readouts \(\rho^{-}\) (left) and \(\rho^{+}\) (right). 
The dashed line is the identity and the solid curve is the corresponding readout. 
The horizontal shaded strip is the interval \([\mar,1-\mar]\) containing \(\supp h\). 
The vertical shaded strip marks the seam interval modified by the readout: 
\([1-\varepsilon,1]\) on the left and \([0,\varepsilon]\) on the right.}
\label{fig:readouts}
\end{figure}

\begin{lemma}[Scalar readouts]
\label{lem:scalar-readouts}
Let \(h\) be a special-hat, and assume
\(0<\varepsilon<\mar\).
Then we have
\begin{equation}
\label{eq:min-identity-parameter}
h(t)=\min\{h(r^-(t)),\,h(r^+(t))\}, \qquad t\in[0,1].
\end{equation}
Equivalently, for every \(n\ge 0\) and every \(x\in[0,1]\), we have
\begin{equation}
\label{eq:min-identity-loop}
h(R^n(x))
=
\min\bigl\{h(\rho^-(z_n(x))),\,h(\rho^+(z_n(x)))\bigr\}.
\end{equation}
\end{lemma}

\begin{proof}
Since \(h\) is a special-hat, we have
\(\supp h\subset[\mar,1-\mar]\).
We prove \eqref{eq:min-identity-parameter} by splitting \([0,1]\) into three regions.

\smallskip
\noindent
\emph{Region 1: \(t\in[\varepsilon,\,1-\varepsilon]\).}
On this interval, both readouts are exact:
\(r^-(t)=t=r^+(t)\) and thus
\[
\min\{h(r^-(t)),h(r^+(t))\}=h(t).
\]

\smallskip
\noindent
\emph{Region 2: \(t\in[0,\varepsilon]\).}
Here \(r^-(t)=t\), while \(r^+(t)\in[\varepsilon,1]\). Since \(\varepsilon<\mar\), we have
\(t<\mar\), and therefore
\(h(t)=0\).
It follows that
\(h(r^-(t))=h(t)=0\),
and hence
\[
\min\{h(r^-(t)),h(r^+(t))\}=0=h(t).
\]

\smallskip
\noindent
\emph{Region 3: \(t\in[1-\varepsilon,1]\).}
Now \(r^+(t)=t\), while \(r^-(t)\in[0,1-\varepsilon]\). Since \(\varepsilon<\mar\), we have
\(t>1-\mar\), and therefore
\(h(t)=0\).
Thus we have
\(h(r^+(t))=h(t)=0\),
and again
\[
\min\{h(r^-(t)),h(r^+(t))\}=0=h(t).
\]

Combining the three regions proves \eqref{eq:min-identity-parameter}. 
Finally, apply \eqref{eq:min-identity-parameter} with \(t=R^n(x)\), and use
\Cref{prop:exact-loop-controller} together with \eqref{eq:rho-readouts-on-loop}:
\[
\rho^\pm(z_n(x))
=
\rho^\pm(E(R^n(x)))
=
r^\pm(R^n(x)),
\]
yielding \eqref{eq:min-identity-loop}.
\end{proof}

We can now state the scalar realization theorem.

\begin{proposition}[Scalar factor realization]
\label{prop:section3-hRn-loop}
Let \(h\) be a special-hat with at most \(m\) breakpoints, and fix
\(0<\varepsilon<\mar\).
Then there exist constants \(C_0,C_1>0\), depending only on \(M\), \(\mar\), \(m\), and
\(\varepsilon\), such that for every \(n\ge 1\), one has
\[
h(R^n(\cdot))\in \Ups_{C_0,C_1n}(\ReLU;1,1).
\]
More precisely, the function \(h(R^n(\cdot))\) is given on \([0,1]\) by the exact identity
\begin{equation}
\label{eq:scalar-loop-realization}
h(R^n(x))
=
\min\bigl\{h(\rho^-(z_n(x))),\,h(\rho^+(z_n(x)))\bigr\},
\qquad x\in[0,1].
\end{equation}
Moreover, the realizing networks may be chosen with weights and biases bounded in absolute value
by a constant \(C_2\), where \(C_2\) depends only on \(M\), \(\mar\), \(m\), and \(\varepsilon\), and is
independent of \(n\).
\end{proposition}

\begin{proof}
The exact identity \eqref{eq:scalar-loop-realization} is precisely
\Cref{lem:scalar-readouts}.

It remains to prove the asserted ReLU-realizability and the coefficient bound.
By \Cref{lem:loop-controller-map}, the embedding \(E\) is a fixed CPwL map \(\R\to\R^2\) and the
controller map \(F\) is a fixed CPwL map \(\R^2\to\R^2\), both with complexity depending only on
\(M\). Therefore the state map
\[
x\longmapsto z_n(x)=F^n(E(x))
\]
belongs to \(\Ups_{C'_0,C'_1n}(\ReLU;1,2)\) for constants \(C'_0,C'_1\) depending only on \(M\).
Since the same fixed map \(F\) is iterated at each stage, the realizing networks for \(z_n\) may be
chosen with weights and biases bounded independently of \(n\).

Next, each readout \(\rho^\pm:\R^2\to[0,1]\) is a fixed CPwL map with finitely many affine pieces,
depending only on \(\varepsilon\). Composing with the fixed CPwL map \(h\) gives two fixed branches
\[
z\longmapsto h(\rho^-(z)),
\qquad
z\longmapsto h(\rho^+(z)),
\]
each of which admits an exact finite ReLU realization with weights and biases bounded by a constant
depending only on \(\mar\), \(m\), and \(\varepsilon\). Hence each branch
\(x\mapsto h(\rho^\pm(z_n(x)))\)
belongs to \(\Ups_{C''_0,C''_1n}(\ReLU;1,1)\), and the weights and biases in such a realization are
bounded independently of \(n\).

Finally, the minimum of two scalar outputs is produced by one additional ReLU layer, for instance via
\[
\min\{u,v\}=v-\ReLU(v-u)=u-\ReLU(u-v).
\]
Running the two branches in parallel increases the width by only a constant factor, and the final
minimum gadget contributes only \(O(1)\) extra width and one extra layer. This proves that
\(h(R^n(\cdot))\in \Ups_{C_0,C_1n}(\ReLU;1,1)\),
and the uniform coefficient bound follows after enlarging the constant if necessary.
\end{proof}

\begin{remark}
\label{rem:no-global-inverse-loop}
The preceding construction uses no global inverse or readout \(\pi\) satisfying
\(\pi(E(t))=t\) on the whole closed loop. Such a global single-valued inverse is impossible because
\(E(0)=E(1)\). The exactness comes instead from the forward loop controller
\(z_m(x)=E(R^m(x))\), while the seam is handled by the two terminal readouts \(\rho^\pm\).
This is precisely the point at which the present loop-controller mechanism differs from the older
scalar-surrogate iteration.
\end{remark}

\subsection{Selector matrices and modified cascade blocks}
\label{subsec:section3-selectors}

We now replace the discontinuous matrix selector \(x\mapsto T_{Q(x)}\) by a CPwL matrix field
read directly from the exact loop state.
Recall from \S\ref{subsec:section3-loop-controller} that for each \(x\in[0,1]\) and each
\(m\ge 0\) we have the exact loop state
\[
z_m(x)=E(R^m(x))\in \Gamma.
\]
Unlike the scalar readouts from \S\ref{subsec:section3-scalar-readouts}, the selector
transition width must depend on the final depth \(n\). The reason is that once a stage enters a
selector transition set, we will propagate that information forward along the exact residual orbit
and force the terminal scalar factor to vanish; for this, the transition intervals must be chosen small
enough to prevent wrap-around across \(M\)-ary cells.

Fix once and for all a number
\(0<\bar\delta<1\),
and for each final depth \(n\ge 1\) set
\begin{equation}
\label{eq:delta-n-selector}
\delta_n:=\bar\delta\,\mar\, M^{-(n+1)}.
\end{equation}
Clearly, we have \(0<\delta_n<\mar\) for every \(n\).
For later reference, we define the {\em transition sets}
\begin{equation}
\label{eq:selector-transition-set}
J_n:=
\bigcup_{k=0}^{M-1}
\Big[
\frac{k}{M},
\frac{k}{M}+\delta_n
\Big]
\subset [0,1].
\end{equation}
Thus \(J_n\) is the union of right neighborhoods of all \(M\)-ary breakpoints,
excluding the seam point \(1\).

For each \(q\in\{0,1,\dots,M-1\}\), choose a continuous piecewise affine function
\(\vartheta_{q}:[0,1]\to[0,1]\)
such that:
\begin{enumerate}[label=(\roman*),leftmargin=3em]
    \item \(\vartheta_{q}(t)=\chi_{[q/M,(q+1)/M]}(t)\) on the complement of \(J_n\);
    \item for every \(t\in[0,1]\), one has
    \[
    \sum_{q=0}^{M-1}\vartheta_{q}(t)=1;
    \]
    \item the endpoint values satisfy
    \[
    \vartheta_{M-1}(0)=\vartheta_{M-1}(1)=1,
    \qquad
    \vartheta_{q}(0)=\vartheta_{q}(1)=0
    \quad (q=0,\dots,M-2),
    \]
    so that the family descends consistently to the closed loop \(\Gamma\).
\end{enumerate}

One may, and we tacitly do, choose these selectors so that each \(\vartheta_{q}\) is supported near
the corresponding \(M\)-ary cell, with the last-cell selector \(\vartheta_{M-1}\) carrying the seam value.
Such a family is easily obtained by taking the exact indicator partition away from \(J_n\), and then
inserting affine transition ramps inside each interval
\([\frac{k}{M},\frac{k}{M}+\delta_n]\) for \(k=0,1,\dots,M-1\).

\begin{figure}[ht]
\centering
\includegraphics[width=.7\textwidth]{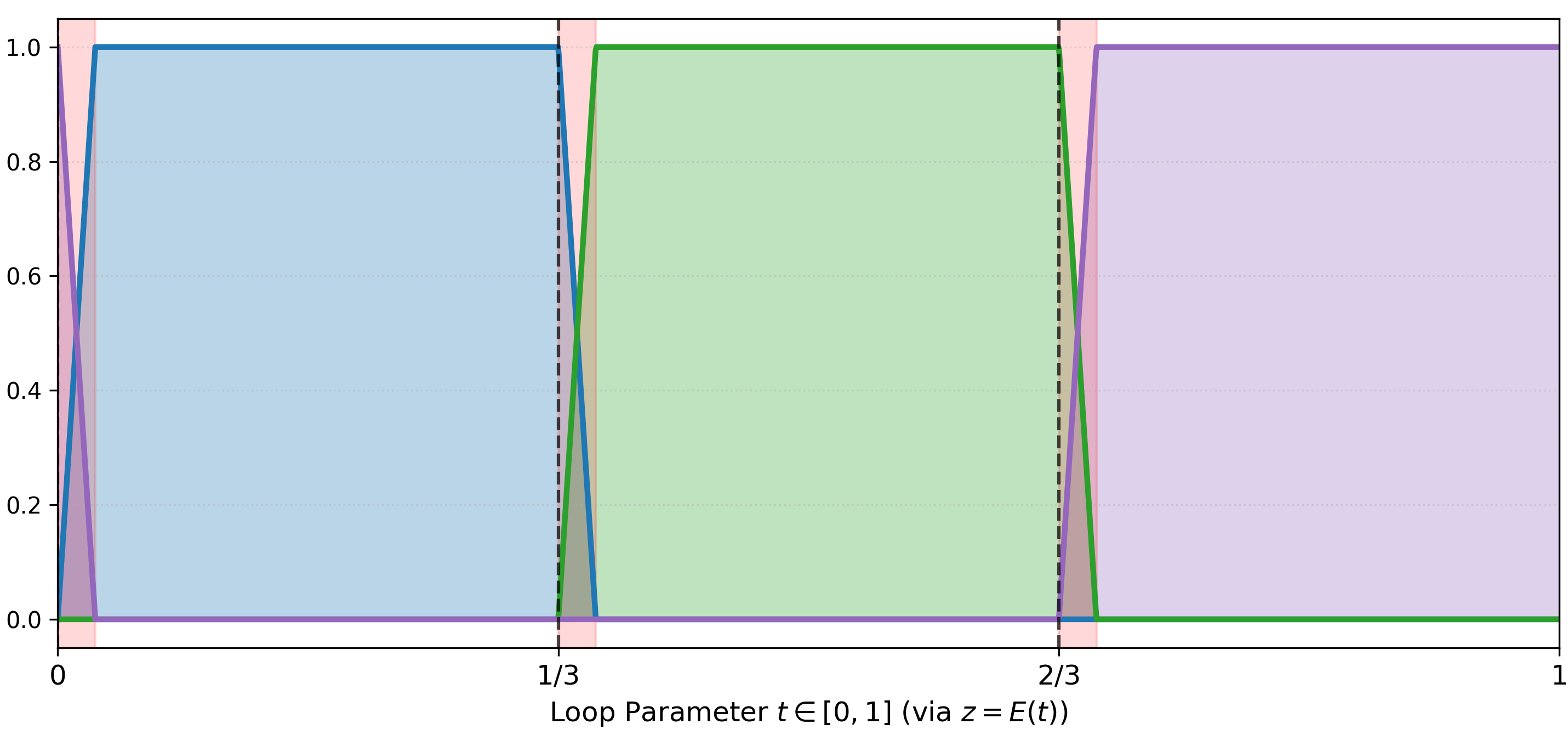}
\caption{Loop-state selector profiles for \(M=3\). Away from the red transition intervals \(J_n\),
the selectors agree with the exact digit indicators; inside \(J_n\), affine ramps produce a continuous
partition of unity. The last selector carries the seam value \(0\sim1\).}
\label{fig:selectors}
\end{figure}

The endpoint compatibility in (iii) allows us to define continuous piecewise affine selector functions
on the loop by
\[
\chi_{q}(E(t)):=\vartheta_{q}(t),\qquad t\in[0,1].
\]
As in the proof of \Cref{lem:loop-controller-map}, we extend these functions arbitrarily but fixedly to
CPwL maps
\[
\chi_{q}:\R^2\to[0,1],
\qquad q=0,1,\dots,M-1.
\]
The functions \(\chi_q\) may be arbitrary away from the loop, subject only to being
CPwL; only their values on \(\Gamma\) are used in the recursion.
Figure~\ref{fig:selectors} shows the resulting selector profiles for $M=3$.

\begin{lemma}[Loop-state selectors]
\label{lem:loop-state-selectors}
The selector family \(\{\chi_{q}\}_{q=0}^{M-1}\) has the following properties.
\begin{enumerate}[label=(\roman*),leftmargin=2em]
    \item For every \(z\in\Gamma\), we have
    \[
    0\le \chi_{q}(z)\le 1
    \qquad\text{and}\qquad
    \sum_{q=0}^{M-1}\chi_{q}(z)=1.
    \]

    \item If \(t\in[0,1]\setminus J_n\), then exactly one selector is equal to \(1\) at \(E(t)\), and it
    is the selector corresponding to the digit \(Q(t)\); equivalently,
    \[
    \chi_{q}(E(t))=\chi_{[q/M,(q+1)/M]}(t),
    \qquad t\in[0,1]\setminus J_n.
    \]

    \item Each \(\chi_{q}\) is CPwL on \(\R^2\), hence belongs to
    \(\Ups_{C_\chi,2}(\ReLU;2,1)\)
    for a constant \(C_\chi\) depending only on \(M\). Moreover, for fixed \(\bar\delta\), the
    realizing networks may be chosen with weights and biases bounded in absolute value by
    \(C_{\chi,\mar} M^n\), where \(C_{\chi,\mar}\) is independent of \(n\).
\end{enumerate}
\end{lemma}

\begin{proof}
Parts (i) and (ii) are immediate from the construction of \(\vartheta_{q}\) on \([0,1]\) and the
definition of \(\chi_{q}\) on \(\Gamma\). Part (iii) follows because each \(\chi_{q}\) is a fixed CPwL
extension of a continuous piecewise affine function on a finite polygonal complex. The only
\(n\)-dependence comes from the transition width \(\delta_n\), so the nonzero slopes of the
transition ramps are bounded by
\[
\frac{C}{\delta_n}
=
\frac{C}{\bar\delta\mar}\,M^{n+1}
\le C_{\chi,\mar}M^n,
\]
after enlarging the constant \(C_{\chi,\mar}\) if necessary.
\end{proof}

We now combine the selectors with the block transition matrices from \S\ref{ss:block-transition}.

\begin{definition}[Modified matrix field]
\label{def:modified-matrix-field-loop}
Define
\begin{equation}
\label{eq:modified-matrix-field-loop}
\widehat T_n(z)
:=
\sum_{q=0}^{M-1}\chi_{q}(z)\,T_q,
\qquad z\in\R^2.
\end{equation}
For \(j=1,\dots,n\), define the stage-\(j\) matrix field on \([0,1]\) by
\begin{equation}
\label{eq:modified-stage-matrices-loop}
\widehat M_j(x):=\widehat T_n(z_{j-1}(x))
=\widehat T_n(E(R^{j-1}(x))).
\end{equation}
\end{definition}

Thus \(\widehat M_j(x)\) is obtained by reading the digit information directly from the exact loop
state \(z_{j-1}(x)\), rather than from an iterated scalar surrogate residual.

For an atomic curve \(\gamma=h e_\mu\) and a coordinate index
\(\ell\in\{1,\dots,pL\}\), we write
\[
g_{n,\ell}(x):=e_\ell^\top G^n(x),
\qquad x\in[0,1].
\]

\begin{proposition}[Modified cascade identity]
\label{prop:modified-cascade-identity-loop}
For every atomic curve \(\gamma=h e_\mu\) and every coordinate
\(\ell\in\{1,\dots,pL\}\), one has
\begin{equation}
\label{eq:modified-cascade-identity-loop}
g_{n,\ell}(x)
=
h(R^n(x))\,e_\ell^\top
\widehat M_1(x)\widehat M_2(x)\cdots \widehat M_n(x)\,u_\mu,
\qquad x\in[0,1].
\end{equation}
\end{proposition}

\begin{proof}
Recall the exact cascade formula
\begin{equation}
\label{eq:exact-cascade-again}
g_{n,\ell}(x)
=
h(R^n(x))\,e_\ell^\top
T_{q_1(x)}T_{q_2(x)}\cdots T_{q_n(x)}\,u_\mu.
\end{equation}
We distinguish two cases.

\smallskip
\noindent
\emph{Case 1 (transition-free): \(R^{j-1}(x)\notin J_n\) for every \(j=1,\dots,n\).}
Then by \Cref{lem:loop-state-selectors}(ii), we have
\[
\chi_{q}(z_{j-1}(x))
=
\chi_{[q/M,(q+1)/M]}(R^{j-1}(x)),
\qquad q=0,\dots,M-1.
\]
Since \(q_j(x)=Q(R^{j-1}(x))\), it follows that
\[
\widehat M_j(x)=T_{q_j(x)},
\qquad j=1,\dots,n.
\]
Substituting this into \eqref{eq:exact-cascade-again} gives
\eqref{eq:modified-cascade-identity-loop}.

\smallskip
\noindent
\emph{Case 2 (transition-hit): \(R^{j-1}(x)\in J_n\) for some \(j\in\{1,\dots,n\}\).}
Let \(m\) be the smallest such $j$. Then
\[
R^{m-1}(x)\in
\Big[
\frac{k}{M},
\frac{k}{M}+\delta_n
\Big]
\]
for some \(k\in\{0,\dots,M-1\}\). Write
\[
R^{m-1}(x)=\frac{k}{M}+\delta,
\qquad 0\le \delta\le \delta_n.
\]

We first treat the exact-breakpoint case \(\delta=0\).
Since \(R^{m-1}(x)=\frac{k}{M}\in[0,1)\), by the endpoint convention for \(Q\) and \(R\),
we have \(R^m(x)=0\).
Hence \(R^j(x)=0\) for every \(j\ge m\), and therefore
\[
h(R^n(x))=h(0)=0.
\]

Now assume \(\delta>0\). Since \(R^{m-1}(x)\) lies just to the right of the breakpoint \(k/M\), its
digit is \(k\), and therefore
\[
R^m(x)=M\delta.
\]
We claim that for every \(s=1,\dots,n-m+1\),
\begin{equation}
\label{eq:no-wrap-iterate}
R^{m-1+s}(x)=M^s\delta.
\end{equation}
Indeed, for such \(s\), we have
\[
M^s\delta
\le
M^{\,n-m+1}\delta_n
\le
M^n\delta_n
=
\bar\delta\,\mar\,M^{-1}
<
\frac1M.
\]
Hence
\(M^s\delta \in \big(0,\frac1M\big)\),
meaning that every intermediate iterate remains in the first \(M\)-ary digit cell. 
Thus \eqref{eq:no-wrap-iterate} follows by induction.

Taking \(s=n-m+1\), we obtain
\[
R^n(x)=M^{\,n-m+1}\delta
\in
[0,M^{\,n-m+1}\delta_n]
\subset
[0,\bar\delta\mar M^{-1}] .
\]
Since \(\bar\delta<1\) and \(M\ge 2\), we have \(\bar\delta M^{-1}<1\), implying that
\(R^n(x)\in [0,\mar)\).
Because \(h\) is a special-hat and \(\supp h\subset[\mar,1-\mar]\), it follows that
\[
h(R^n(x))=0.
\]

Thus in all subcases of Case 2, the left-hand side of
\eqref{eq:modified-cascade-identity-loop} vanishes.
The right-hand side also vanishes, because it contains the same scalar factor \(h(R^n(x))\).
Hence \eqref{eq:modified-cascade-identity-loop} holds in this case as well.

Combining the two cases completes the proof.
\end{proof}

\begin{figure}[ht]
\centering
\includegraphics[width=.8\textwidth]{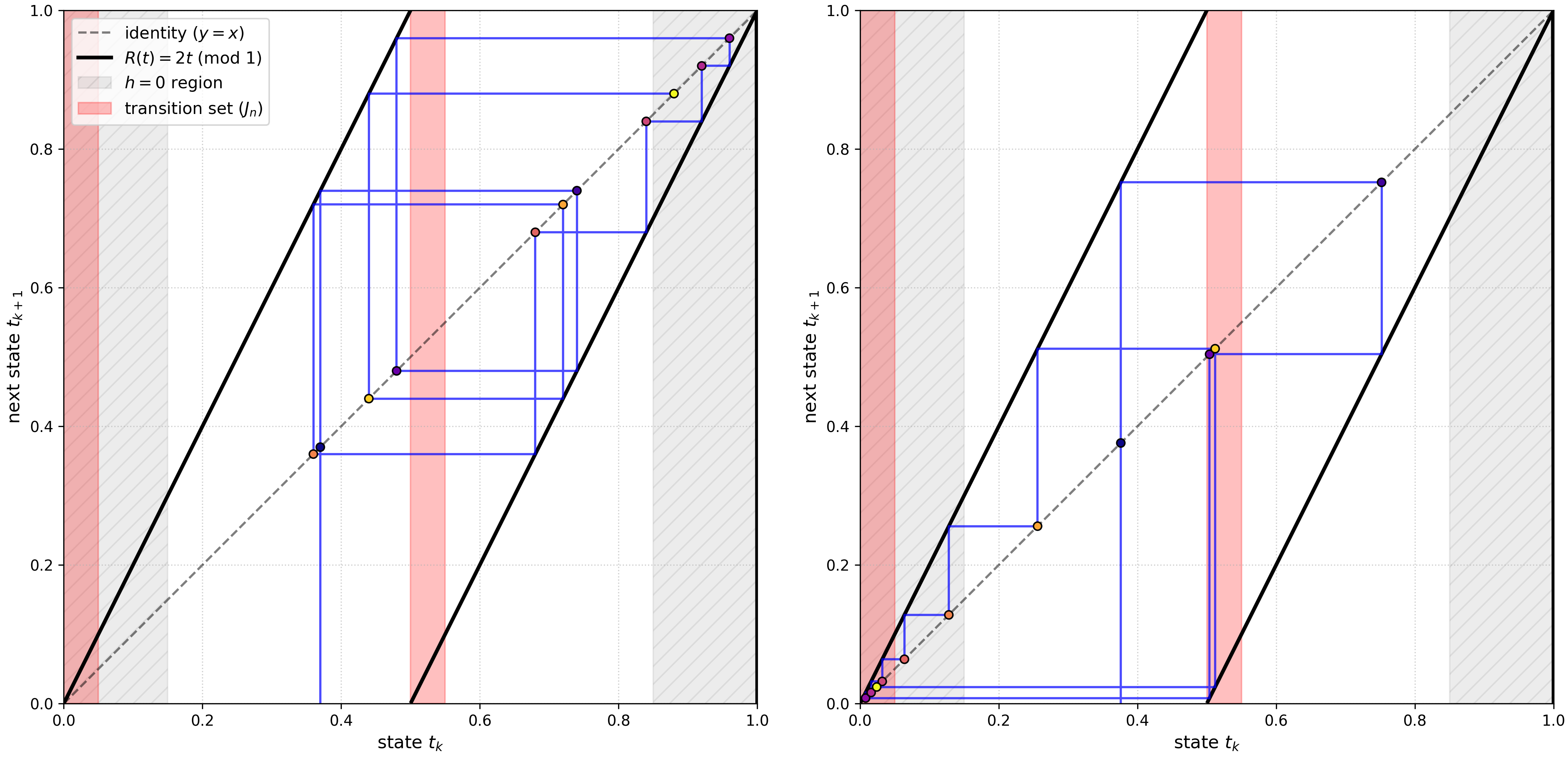}
\caption{Residual dynamics of the exact loop controller for \(M=2\).
Left: the first \(n\) iterates avoid the transition set \(J_n\).
Right: an orbit enters the transition interval near \(t=1/2\), and the choice of \(\delta_n\)
ensures \(R^n(x)\) still lies in the zero region of \(h\). The cobweb is continued past time \(n\)
only to show the later residual dynamics.}
\label{fig:cobweb}
\end{figure}

\begin{remark}
\label{rem:selector-section-philosophy}
The point of \Cref{prop:modified-cascade-identity-loop} is that the forward controller remains exact
throughout: all residual information is carried by the exact loop states \(z_j(x)=E(R^j(x))\).
The selector ambiguity is purely stage-local and confined to the transition set \(J_n\). Once some
stage enters \(J_n\), the choice of the width \(\delta_n\) forces the exact residual orbit into the
terminal zero region of the special-hat without wrap-around across \(M\)-ary cells, so the common
scalar factor \(h(R^n(x))\) vanishes. Thus the ambiguity is absorbed only at the terminal scalar
stage, not by any globally exact selector on the loop, and no separate cumulative transition-set
bookkeeping is needed.
Figure~\ref{fig:cobweb} illustrates these two alternatives in the binary case:
transition-free residual transport up to the terminal depth on the one hand, and a transition hit
whose terminal residual still lies in the zero region of the special hat on the other.
\end{remark}

\subsection{Recursive realization on the unit interval}
\label{subsec:section3-unit-interval}

We continue with the same fixed support margin \(\mar\in(0,\tfrac12)\), the same terminal-readout
parameter \(\varepsilon\in(0,\mar)\) from \S\ref{subsec:section3-scalar-readouts}, and the same selector
parameter \(\bar\delta\in(0,1)\) from \S\ref{subsec:section3-selectors}, with
\(\delta_n=\bar\delta\,\mar\,M^{-(n+1)}\).
Thus the scalar factor \(h(R^n(x))\) is realized using the fixed readouts \(\rho^\pm\), while the
modified cascade matrices \(\widehat M_j(x)\) are built from the level-\(n\) selector family
\(\{\chi_{q}\}_{q=0}^{M-1}\).

We now implement the modified cascade identity by a fixed-depth recursive block driven by the
exact loop controller.
We first recall the following product gadget from \cite[Lemma 9]{source}.

\begin{lemma}[Product gadget]
\label{lem:product-gadget}
Let \(N\in\N\) and \(a>0\). Define
\[
\Pi_a(\lambda,y)
:=
-\ReLU(\lambda a\mathbf 1-y)
-\ReLU((1-\lambda)a\mathbf 1-\ReLU(-y))
+a\mathbf 1,
\]
for \(\lambda\in[0,1]\) and \(y\in\R^N\), where \(\mathbf 1\in\R^N\) is the vector of all ones.
Then
\[
\Pi_a\in \Ups_{2N+1,2}(\ReLU;N+1,N),
\]
and for every \(y\in[-a,a]^N\) and every \(\lambda\in[0,1]\) one has
\[
\Pi_a(1,y)=y,\qquad \Pi_a(0,y)=0,\qquad \Pi_a(\lambda,0)=0.
\]
\end{lemma}

Fix a coordinate index \(\ell\in\{1,\dots,pL\}\),
and recall from \Cref{prop:modified-cascade-identity-loop} that
\[
g_{n,\ell}(x)
=
h(R^n(x))\, e_\ell^\intercal \widehat M_1(x)\cdots \widehat M_n(x)\,u_\mu,
\qquad x\in[0,1].
\]
It is therefore natural to work with the corresponding column-vector states
\[
\Phi_0(x):=h(R^n(x))\,e_\ell,
\qquad
\Phi_j(x):=\widehat M_j(x)^\intercal \Phi_{j-1}(x),
\qquad j=1,\dots,n,
\]
which give
\[
g_{n,\ell}(x)=u_\mu^\intercal \Phi_n(x).
\]
The role of the initial factor \(h(R^n(x))\) is crucial: by \Cref{prop:modified-cascade-identity-loop},
if some stage enters the selector transition region, then \(h(R^n(x))=0\), and therefore the entire
adjoint state recursion starts from \(0\); this is exactly the situation in which the product gadget
uses the identity \(\Pi_a(\lambda,0)=0\).

Furthermore, choose a constant \(a>0\) so large that
\[
\|T_q^\intercal y\|_\infty\le a
\]
for every \(q\in\{0,\dots,M-1\}\) and every intermediate vector \(y=\Phi_j(x)\) arising in the
recursion.
This is possible because only finitely many matrices \(T_q\) occur and all states are uniformly
bounded.

We now define the network-produced states \(\widehat \Phi_j\) and the loop-controller states
\(\widehat z_j\) recursively by
\begin{equation}
\label{eq:recursive-update-block-loop}
\left\{
\begin{aligned}
\widehat \Phi_j(x)
&=
\sum_{q=0}^{M-1}
\Pi_a\bigl(\chi_{q}(\widehat z_{j-1}(x)),\,T_q^\intercal \widehat \Phi_{j-1}(x)\bigr),\\
\widehat z_j(x)
&=
F(\widehat z_{j-1}(x)),
\end{aligned}
\right.
\end{equation}
for \(j=1,\dots,n\), with
\[
\widehat \Phi_0(x):=
h(R^n(x))\,e_\ell,
\qquad
\widehat z_0(x):=E(x).
\]
By construction of the modified matrix field,
\[
\widehat M_j(x)=\widehat T_n(z_{j-1}(x))
=
\sum_{q=0}^{M-1}\chi_{q}(z_{j-1}(x))\,T_q,
\qquad j=1,\dots,n.
\]

\begin{lemma}[Exactness of the recursive block]
\label{lem:block-exactness}
For every \(j=0,\dots,n\) and every \(x\in[0,1]\), the states produced by
\eqref{eq:recursive-update-block-loop} satisfy
\[
\widehat z_j(x)=z_j(x)=E(R^j(x)),
\]
and
\[
\widehat \Phi_j(x)=\Phi_j(x)
=
\widehat M_j(x)^\intercal\cdots \widehat M_1(x)^\intercal h(R^n(x))\,e_\ell.
\]
In particular,
\[
g_{n,\ell}(x)=u_\mu^\intercal \widehat \Phi_n(x).
\]
\end{lemma}

\begin{proof}
The claim for the loop state is immediate from the recursion and
\Cref{lem:loop-controller-map,prop:exact-loop-controller}. Indeed, \(\widehat z_0(x)=E(x)=z_0(x)\),
and if \(\widehat z_{j-1}(x)=z_{j-1}(x)\), then
\[
\widehat z_j(x)=F(\widehat z_{j-1}(x))=F(z_{j-1}(x))=z_j(x).
\]

For the vector state, the statement is true for \(j=0\) by definition.
Assume it holds at level \(j-1\).

If \(R^{m-1}(x)\notin J_n\) for every \(m=1,\dots,n\), then by
\Cref{lem:loop-state-selectors} we have
\[
\chi_{q}(z_{j-1}(x))
=
\chi_{[q/M,(q+1)/M]}(R^{j-1}(x)),
\qquad q=0,\dots,M-1,
\]
so exactly one selector is equal to \(1\), namely the one corresponding to \(q_j(x)\).
Hence \Cref{lem:product-gadget} gives
\[
\widehat \Phi_j(x)
=
T_{q_j(x)}^\intercal \widehat \Phi_{j-1}(x)
=
\widehat M_j(x)^\intercal \widehat \Phi_{j-1}(x),
\]
which is the desired recursion.

If \(R^{m-1}(x)\in J_n\) for some \(m\in\{1,\dots,n\}\), then by the proof of
\Cref{prop:modified-cascade-identity-loop} one has
\[
h(R^n(x))=0.
\]
Therefore \(\widehat \Phi_0(x)=0\), and by induction \(\widehat \Phi_{j-1}(x)=0\).
Since \(\Pi_a(\lambda,0)=0\) for every \(\lambda\in[0,1]\), every term in
\eqref{eq:recursive-update-block-loop} vanishes, and thus
\[
\widehat \Phi_j(x)=0.
\]
This agrees with the desired formula because the same scalar factor \(h(R^n(x))\) appears there.

The final identity \(g_{n,\ell}(x)=u_\mu^\intercal \widehat \Phi_n(x)\) now follows immediately
from the definition of \(\Phi_n\).
\end{proof}

We can now conclude the unit-interval theorem.

\begin{theorem}[Vectorized atomic realization on the unit interval]
\label{thm:section3-special-unit-interval}
Let \(\gamma(t)=h(t)e_\mu\) be an atomic curve, and assume that \(h\) has at most \(m\)
breakpoints.  Fix \(0<\varepsilon<\mar\) and \(0<\bar\delta<1\).  Then there exist
constants \(C_0,C_1>0\), depending only on \(M,\mar,\varepsilon,\bar\delta,L,p\), and \(m\),
such that
\[
\vect(V^n\gamma) \in\Ups_{C_0,C_1n}(\ReLU;1,pL),
\qquad n\ge1.
\]
Moreover, for the fixed function \(h\) and fixed matrices \(A_j\), the realizing networks may be
chosen with weights and biases bounded by \(C_2\Lambda^n\), where \(C_2,\Lambda>0\) depend on
the fixed CPwL data of \(h\), on the matrices \(A_j\), and on the preceding fixed parameters.
\end{theorem}

\begin{proof}
For each coordinate \(\ell=1,\dots,pL\), \Cref{prop:section3-hRn-loop} realizes the initial scalar factor
\(h(R^n)\) on \([0,1]\) with fixed width and depth \(O(n)\),
and with weights and biases bounded independently of \(n\).
Each recursive step \eqref{eq:recursive-update-block-loop} has width depending only on \(M\) and \(pL\),
and depth bounded independently of \(n\). Concatenating \(n\) such blocks therefore produces
\(g_{n,\ell}\) with fixed width and depth linear in \(n\).

Since the matrices \(T_q\) are fixed, there exists \(\Lambda_0>0\), depending only on the operator
data, such that every intermediate state in the recursion is bounded by \(C\Lambda_0^n\). Hence \(a\)
may be chosen so that
\[
a\le C'\Lambda_0^n
\]
for a constant \(C'\) independent of \(n\). The product gadget \(\Pi_a\) then has coefficients
bounded by \(C''\Lambda_0^n\),
while the selector and loop-controller branches contribute at most
the exponential bound coming from the selector width \(\delta_n\) and the fixed controller map.
Therefore each coordinate branch
\(g_{n,\ell}\) is realized by a network whose weights and biases are bounded by
\(C_2\Lambda^n\) for suitable constants \(C_2,\Lambda>0\) depending only on the stated data and the
fixed matrices \(A_j\).

Running the \(pL\) coordinate branches in parallel increases the width by at most a factor of \(pL\),
while preserving the linear-in-\(n\) depth bound and the same coarse exponential coefficient bound
after enlarging \(C_2\) if necessary. Thus \(\vect(V^n\gamma)=G^n\) belongs to
\(\Ups_{C_0,C_1n}(\ReLU;1,pL)\), with the stated weight and bias bound.
\end{proof}

\subsection{From atomic curves to general CPwL curves}
\label{subsec:section3-global}

We next pass from blockwise realizations on $[0,1]$ to a realization on the whole line.
For $k=1,\dots,L$, let $G^{n,k}:[0,1]\to\R^p$ denote the $k$th block of $G^n=\vect(V^n\gamma)$.
Then, for $t\in[k-1,k]$, we have
\[
V^n\gamma(t)=G^{n,k}(t-k+1).
\]
Define the ramp functions
\begin{equation}
\label{eq:ramp-functions}
\sigma_k(t)=\ReLU(t-k+1)-\ReLU(t-k),\qquad k=1,\dots,L.
\end{equation}

\begin{lemma}[Gluing lemma]
\label{lem:gluing}
Let $f_k:[0,1]\to\R^p$ be continuous functions satisfying
\[
f_1(0)=0,\quad
f_L(1)=0,\quad\text{and}\quad
f_k(1)=f_{k+1}(0)\qquad\text{for} \quad k=1,\dots,L-1 .
\]
Then the function
\begin{equation}
\label{eq:gluing-formula}
F(t)
=
f_1(\sigma_1(t))
+
\sum_{k=2}^L \bigl(f_k(\sigma_k(t))-f_k(0)\bigr)
\end{equation}
satisfies
\[
F(t)=f_k(t-k+1),\qquad t\in[k-1,k],\quad k=1,\dots,L,
\]
and vanishes outside $[0,L]$.
\end{lemma}

\begin{proof}
Fix $t\in[k-1,k]$.  We have
\[
\sigma_j(t)=
\begin{cases}
1, & j<k,\\
t-k+1, & j=k,\\
0, & j>k.
\end{cases}
\]
Substituting into \eqref{eq:gluing-formula} gives
\[
F(t)
=
f_1(1)+\sum_{j=2}^{k-1}\bigl(f_j(1)-f_j(0)\bigr)+\bigl(f_k(t-k+1)-f_k(0)\bigr).
\]
By the endpoint matching conditions this telescopes to
\[
F(t)=f_k(t-k+1).
\]
For $t\le 0$, all $\sigma_k(t)=0$, so $F(t)=f_1(0)=0$.
For $t\ge L$, all $\sigma_k(t)=1$, so the same telescoping gives $F(t)=f_L(1)=0$.
\end{proof}

Applying this lemma to the blocks $G^{n,k}$ yields the global realization for atomic curves.

\begin{theorem}[Atomic curve on \(\R\)]
\label{thm:section3-special-R}
Let \(\gamma=h e_\mu\) be an atomic curve, and assume that \(h\) has at most \(m\) breakpoints.
Fix
\(0<\varepsilon<\mar\) and \(0<\bar\delta<1\).
Then there exist constants \(C_0,C_1>0\), depending only on \(M\), \(\mar\), \(\varepsilon\),
\(\bar\delta\), \(p\), \(L\), and \(m\), such that
\[
V^n\gamma\in \Ups_{C_0,C_1 n}(\ReLU;1,p),
\qquad n\ge 1.
\]
Moreover, for each fixed \(h\) and fixed matrices \(A_j\), the realizing networks may be chosen
with weights and biases bounded in absolute value by \(C_2\Lambda^n\), where \(C_2,\Lambda>0\)
depend on the CPwL data of \(h\), on the matrices \(A_j\), and on the preceding fixed parameters.
\end{theorem}

\begin{proof}
By \Cref{thm:section3-special-unit-interval}, each block \(G^{n,k}\) is realized on \([0,1]\)
by a vector-valued ReLU network of width bounded independently of \(n\), depth \(O(n)\), and with
weights and biases bounded by \(C_2\Lambda^n\) for suitable constants \(C_2,\Lambda\) depending only
on the stated data and the fixed matrices \(A_j\).

The global curve \(V^n\gamma\) is obtained by gluing these blockwise realizations over the unit
intervals \([k-1,k]\), \(k=1,\dots,L\), via \Cref{lem:gluing}. The ramp functions \(\sigma_k\) require one
additional ReLU layer and have coefficients bounded independently of \(n\). Since the gluing formula
\eqref{eq:gluing-formula} uses only finitely many affine shifts and finite sums, the standard closure
properties of the network classes imply that
\[
V^n\gamma\in \Ups_{C_0,C_1n}(\ReLU;1,p)
\]
for suitable constants \(C_0,C_1\) with the stated dependence. Since the gluing step introduces only
\(O(1)\) additional affine operations with \(n\)-independent coefficients, the same coarse exponential
bound \(C_2\Lambda^n\) on the weights and biases is preserved after enlarging \(C_2\) if necessary.
\end{proof}

Finally we pass from atomic curves to arbitrary compactly supported CPwL curves.
By the atomic-curve decomposition from \S\ref{ss:special-hat-decomp},
every compactly supported CPwL curve \(\gamma:\R\to\R^p\) with \(\supp(\gamma)\subset[0,L]\)
can be written, after refining the breakpoint set if necessary, in the form
\begin{equation}
\label{eq:finite-hat-decomposition-curve}
\gamma(t)
=
\sum_{\mu=1}^p \sum_{\nu=1}^{N_\mu} a_{\mu,\nu}\, h_{\mu,\nu}(t-\delta_{\mu,\nu})\,e_\mu,
\end{equation}
where \(a_{\mu,\nu}\in\R\), \(\delta_{\mu,\nu}\in\R\), and each \(h_{\mu,\nu}\) is a special hat with at most
three breakpoints.

By translation covariance from \S\ref{ss:translation-covariance}, we have
\[
V^n\bigl(h(\cdot-\delta)e_\mu\bigr)(t)
=
V^n(he_\mu)(t-M^{-n}\delta).
\]
Thus each summand in \eqref{eq:finite-hat-decomposition-curve} is obtained from the atomic-curve
case by an affine input shift, and the general case follows by finite summation.

\begin{theorem}[Homogeneous main theorem]
\label{thm:section3-homogeneous-main}
Let \(\gamma:\mathbb R\to\mathbb R^p\) be a compactly supported CPwL curve with
\(\supp\gamma\subset[0,L]\), and assume that \(V\) preserves the support window \([0,L]\).
If \(\gamma\) has at most \(m\) breakpoints, then 
there exist constants \(C_0,C_1>0\), depending only on \(M,p,L\), and \(m\), such that
\[
V^n\gamma\in\Ups_{C_0,C_1n}(\ReLU;1,p),
\qquad n\ge1.
\]
Moreover, for each fixed \(\gamma\) and fixed matrices \(A_j\), the realizing networks may be
chosen with weights and biases bounded by \(C_2\Lambda^n\), for constants \(C_2,\Lambda>0\)
depending on the fixed CPwL data of \(\gamma\), the matrices \(A_j\), and the fixed construction
parameters.
\end{theorem}

\begin{proof}
Decompose \(\gamma\) as in \eqref{eq:finite-hat-decomposition-curve}.
Apply \Cref{thm:section3-special-R} to each atomic summand, then use translation covariance
and finite summation.

Since each \(h_{\mu,\nu}\) has uniformly bounded breakpoint complexity, the constants in
\Cref{thm:section3-special-R} may be chosen uniformly over all summands. Moreover, the number of
summands in \eqref{eq:finite-hat-decomposition-curve} is bounded in terms of \(p\) and \(m\).
Affine input shifts do not change the asymptotic width/depth bounds or the coarse exponential bound
on weights and biases, and finite summation preserves these bounds after enlarging the constants if
necessary. This proves the claim.
\end{proof}

\section{Corollaries and applications}
\label{sec:corollaries-applications}

This section records several consequences of \Cref{thm:intro-homogeneous} and then turns to
geometric applications. Some reductions remain within the linear-depth scope of the homogeneous
theorem, notably homogeneous anchored constructions and homogeneous finite-state systems.
Other extensions, including affine and stage-dependent forcing rules, are obtained by weaker
black-box reductions from the homogeneous theorem. We begin with this stage-dependent forcing
result and then turn to anchored profiles, finite-state systems, and geometric examples.

\subsection{Stage-dependent forcing}
\label{ss:stage-dependent}

We first record a completely general black-box extension of the homogeneous theorem to affine and
stage-dependent forcing rules. No new controller mechanism is introduced here. Instead, one
expands the iterate into a finite sum of homogeneous terms, applies the homogeneous realization
theorem term by term, and then accumulates the outputs by finite summation. The price for this
generality is that the resulting depth bound is quadratic, \(O(n^2)\), rather than linear.

The ordinary affine case
\[
W\gamma=V\gamma+B
\]
is included as the special case \(B_r\equiv B\) of the stage-dependent family
\[
W_r\gamma=V\gamma+B_r.
\]
Thus no separate affine theorem is needed at this level. Even in the scalar case, however, this
affine viewpoint already enlarges the class of recursive functions that can be brought into the
framework: it opens the possibility of treating classical self-affine constructions such as
Takagi--Knopp-type functions \cite{takagi,knopp} and fractal interpolation functions \cite{barnsley}.

We begin by recording the elementary iterate formula behind this reduction.

\begin{lemma}[Iterate formula for stage-dependent forcing]
\label{lem:stage-dependent-iterate}
Let \(B_r:\R\to\R^p\) be compactly supported CPwL curves, and define
\begin{equation}
\label{eq:stage-dependent-map}
W_r\gamma := V\gamma + B_r,\qquad r\ge 0.
\end{equation}
Then, for every \(n\ge 1\), we have
\begin{equation}
\label{eq:stage-dependent-iterate}
W_{n-1}\cdots W_1W_0\gamma
=
V^n\gamma + \sum_{r=0}^{n-1} V^{\,n-1-r} B_r.
\end{equation}
\end{lemma}

\begin{proof}
This is immediate by induction on \(n\). The case \(n=1\) is just
\[
W_0\gamma = V\gamma + B_0.
\]
If \eqref{eq:stage-dependent-iterate} holds for some \(n\), then
\[
W_nW_{n-1}\cdots W_0\gamma
=
V\!\left(V^n\gamma + \sum_{r=0}^{n-1}V^{\,n-1-r}B_r\right)+B_n
=
V^{n+1}\gamma + \sum_{r=0}^{n}V^{\,n-r}B_r,
\]
which is the same formula with \(n+1\) in place of \(n\).
\end{proof}

The preceding lemma shows that stage-dependent forcing rules reduce to a finite sum of homogeneous
iterates. We may therefore apply the homogeneous realization theorem term by term and then
accumulate the resulting outputs by finite summation. This gives a weaker but completely general
consequence of the homogeneous theory, with quadratic depth rather than linear depth.

\begin{theorem}[Stage-dependent forcing theorem]
\label{thm:black-box-stage-dependent}
Let \(\gamma:\R\to\R^p\) be a compactly supported CPwL curve, and for each \(r\ge 0\) let
\(B_r:\R\to\R^p\) be a compactly supported CPwL curve.
Assume that there exist \(L\ge 1\) and \(m\ge 0\) such that
\[
\supp(\gamma)\subset [0,L],\qquad \supp(B_r)\subset [0,L]\quad\text{for all }r\ge 0,
\]
and each \(B_r\) has at most \(m\) breakpoints.
Let $W_r$ be as in \eqref{eq:stage-dependent-map}.
Then there exist constants \(C_0,C_1>0\), depending only on \(M\), \(p\), \(L\), the breakpoint
data of \(\gamma\), and the uniform breakpoint bound \(m\), such that
\[
W_{n-1}\cdots W_1W_0\gamma
\in
\Ups_{C_0,C_1 n^2}(\ReLU;1,p),
\qquad n\ge 1.
\]
In particular, both affine rules \(W\gamma=V\gamma+B\) and stage-dependent rules
\(W_r\gamma=V\gamma+B_r\) with uniformly controlled compactly supported CPwL forcing terms admit
exact fixed-width ReLU realizations with quadratic depth \(O(n^2)\).
\end{theorem}

\begin{proof}
By \eqref{eq:stage-dependent-iterate}, the iterate \(W_{n-1}\cdots W_1W_0\gamma\) is a sum of one
homogeneous term \(V^n\gamma\) and \(n\) further homogeneous iterates \(V^{\,n-1-r}B_r\),
\(r=0,\dots,n-1\). The first term is realized exactly by \Cref{thm:intro-homogeneous} with fixed
width and depth \(O(n)\). Since each \(B_r\) is supported in \([0,L]\) and has at most \(m\)
breakpoints, the same theorem applies uniformly to every summand \(V^{\,n-1-r}B_r\), with depth
\(O(n-r)\).

We then realize the whole sum by serial concatenation: carry the input \(x\) through a source channel,
maintain a running \(p\)-vector accumulator, and add the summands one by one by affine updates.
Because each constituent network has width bounded independently of \(n\), the total width remains
bounded by a constant depending only on the stated data. The total depth is the sum of the depths
of the constituent realizations, and is therefore \(O(n^2)\). The output agrees exactly with
\(W_{n-1}\cdots W_1W_0\gamma\) by \eqref{eq:stage-dependent-iterate}. This establishes the proof.
\end{proof}

\subsection{Open curves via anchored profiles}
\label{subsec:anchored-profiles}

We now explain how anchored open-curve constructions fit into the same framework.
The guiding idea is to subtract a reference profile and work instead with the resulting compactly
supported defect. At the level of the one-step reduction below, affine and stage-dependent rules are
treated in exactly the same way: in both cases one obtains a compactly supported defect recursion
with a compactly supported forcing term. Upon iteration, the homogeneous case falls under
\Cref{thm:intro-homogeneous}, while the affine and stage-dependent cases fall under the
stage-dependent forcing theorem from \S\ref{ss:stage-dependent}.

The key point is that compact support is needed only for the \emph{defect} relative to a fixed
reference profile.
Let
\[
W\gamma=V\gamma+B,
\]
where \(V\) is the homogeneous refinement operator from Section~\ref{sec:homogeneous-extension}.
Write
\[
\gamma=\Gamma+\eta
\]
for a reference profile \(\Gamma\) and a defect \(\eta\).
Then
\(W\gamma=W\Gamma+V\eta\), and so if one sets
\[
E:=W\Gamma-\Gamma,
\]
then
\[
W\gamma=\Gamma+\bigl(V\eta+E\bigr).
\]
This suggests that if \(\eta\) and \(E\) are compactly supported and CPwL, then the iterate problem
for an open curve should reduce directly to \Cref{thm:black-box-stage-dependent}.

\begin{definition}
\label{def:anchored-profile}
A CPwL curve \(\Gamma:\R\to\R^p\) is called an \emph{anchor profile} if it is eventually constant at
both ends, i.e. there exist \(\Gamma_-,\Gamma_+\in\R^p\) and \(t_*>0\) such that
\[
\Gamma(t)=\Gamma_- \quad (t\le -t_*),
\qquad
\Gamma(t)=\Gamma_+ \quad (t\ge t_*).
\]
A curve \(\gamma:\R\to\R^p\) is said to be \emph{anchored by \(\Gamma\)} if
\(\gamma-\Gamma\)
is compactly supported and CPwL.
\end{definition}

Thus the anchor carries the endpoint data, while the compactly supported defect carries the local
geometry.

\begin{lemma}[Affine anchor-mismatch criterion]
\label{lem:affine-anchor-mismatch-criterion}
Let
\[
W\gamma=V\gamma+B
\]
with
\[
(V\gamma)(t)=\sum_{j\in\Z}A_j\gamma(Mt-j),
\qquad
S:=\sum_{j\in\Z}A_j.
\]
Let \(\Gamma:\R\to\R^p\) be CPwL and eventually constant at both ends:
\[
\Gamma(t)=\Gamma_- \quad (t\ll 0),
\qquad
\Gamma(t)=\Gamma_+ \quad (t\gg 0).
\]
Assume also that \(B:\R\to\R^p\) is CPwL and eventually constant at both ends:
\[
B(t)=B_- \quad (t\ll 0),
\qquad
B(t)=B_+ \quad (t\gg 0).
\]
Then the affine anchor mismatch
\(E=W\Gamma-\Gamma\)
is CPwL. Moreover, \(E\) is compactly supported if and only if
\begin{equation}
\label{eq:affine-anchor-endpoint-condition}
S\Gamma_-+B_-=\Gamma_-,
\qquad
S\Gamma_++B_+=\Gamma_+.
\end{equation}
\end{lemma}

\begin{proof}
Since \(\Gamma\) and \(B\) are CPwL, so is \(E=W\Gamma-\Gamma=V\Gamma+B-\Gamma\).
To determine whether \(E\) is compactly supported, we examine its tails.

For \(t\ll 0\), every argument \(Mt-j\) is also far to the left, hence
\[
\Gamma(Mt-j)=\Gamma_-
\qquad\text{for all }j,
\]
and therefore
\[
(V\Gamma)(t)=\sum_{j\in\Z}A_j\Gamma_-=S\Gamma_-.
\]
Since also \(B(t)=B_-\) for \(t\ll 0\), we obtain
\[
E(t)=S\Gamma_-+B_- - \Gamma_-
\qquad (t\ll 0).
\]
Similarly, for \(t\gg 0\), we have
\[
E(t)=S\Gamma_++B_+ - \Gamma_+
\qquad (t\gg 0).
\]
Thus \(E\) has constant left and right tails, and it is compactly supported if and only if both tails
vanish, which is exactly \eqref{eq:affine-anchor-endpoint-condition}.
\end{proof}

\begin{remark}
\label{rem:affine-anchor-fixed-space}
The anchored-profile mechanism is therefore natural precisely when the endpoint values lie in the
affine fixed-point sets determined by the tail values of \(B\):
\[
\{v\in\R^p:\ Sv+B_-=v\},
\qquad
\{v\in\R^p:\ Sv+B_+=v\}.
\]
If \(B\) is compactly supported, these reduce to the ordinary fixed-point space
\[
\Fix(S)=\{v\in\R^p:\ Sv=v\}.
\]
In particular, if \(S=I_p\), then every eventually constant anchor profile is compatible with every
compactly supported forcing term.
\end{remark}

\begin{proposition}[Affine anchored-profile reduction]
\label{prop:affine-anchored-profile-reduction}
Let
\[
W\gamma=V\gamma+B
\]
be the affine refinement operator on \(\R^p\)-valued curves.
Assume there exists an anchor profile \(\Gamma:\R\to\R^p\) such that
\(E=W\Gamma-\Gamma\)
is compactly supported and CPwL.
Let
\[
\gamma=\Gamma+\eta,
\]
where \(\eta:\R\to\R^p\) is compactly supported and CPwL.
Then for every \(n\ge1\), we have
\begin{equation}
\label{eq:affine-anchored-reduction}
W^n\gamma=\Gamma+\bar W^{\,n}\eta ,
\end{equation}
where 
\(\bar W\xi:=V\xi+E\)
is an auxiliary affine operator.
\end{proposition}

\begin{proof}
Since
\[
W(\Gamma+\xi)=W\Gamma+V\xi=\Gamma+E+V\xi=\Gamma+\bar W\xi,
\]
we obtain
\[
W(\Gamma+\eta)=\Gamma+\bar W\eta.
\]
Applying the same identity inductively gives
\[
W^n\gamma=W^n(\Gamma+\eta)=\Gamma+\bar W^{\,n}\eta
\]
for every \(n\ge1\).
\end{proof}

\begin{corollary}[Open-curve corollary]
\label{cor:affine-anchored-profile-cor}
Under the assumptions of \Cref{prop:affine-anchored-profile-reduction}, if the stage-dependent
forcing theorem of \S\ref{ss:stage-dependent} applies to the compactly supported defect \(\eta\) and
the compactly supported forcing term \(E\), then for every \(n\ge 1\) the iterate \(W^n\gamma\)
admits a fixed-width ReLU realization with quadratic depth \(O(n^2)\). In particular, this applies
to the ordinary affine case as the special case of constant forcing.
\end{corollary}

\begin{proof}
This is immediate from \eqref{eq:affine-anchored-reduction}, upon applying
\Cref{thm:black-box-stage-dependent} to the auxiliary map \(\bar W\). This proves the claim.
\end{proof}

\begin{remark}
The affine anchor mismatch always decomposes as
\[
E=W\Gamma-\Gamma=(V\Gamma-\Gamma)+B=D+B,
\]
where
\[
D=V\Gamma-\Gamma.
\]
Thus the homogeneous anchor mismatch \(D\) acts as an additional forcing term.
\end{remark}

\subsection{Finite-state refinement, stacking, and readouts}
\label{subsec:finite-state-stacking}

Certain recursive constructions are not closed under a single stateless rule.  Instead, one must
remember a finite label describing the current type of the piece being refined; the Hilbert curve is
the canonical example, cf. \cite{Hilbert}.  The point of the present subsection is that such finite-state systems are not
outside the affine framework.  They are ordinary affine refinement operators written in a basis
adapted to a finite state space.

Let \(\cS=\{1,\dots,r\}\) be a finite state set, and let
\[
\Gamma=(\gamma_a)_{a\in\cS},
\qquad
\gamma_a:\R\to\R^p.
\]
For the Hilbert example, we would have $p=2$ and $r=4$.
Fix \(M\ge2\), matrices
\[
A_j^{ab}\in\R^{p\times p},
\qquad
a,b\in\cS,\quad j\in\Z,
\]
with \(A_j^{ab}=0\) for all but finitely many \(j\), and forcing terms
\[
B_a:\R\to\R^p,
\qquad a\in\cS.
\]
Define the finite-state affine refinement operator by
\[
(\mathfrak W\Gamma)_a(t)
=
\sum_{j\in\Z}\sum_{b=1}^r
A_j^{ab}\,\gamma_b(Mt-j)+B_a(t),
\qquad a=1,\dots,r.
\]
Its homogeneous part is
\[
(\mathfrak V\Gamma)_a(t)
=
\sum_{j\in\Z}\sum_{b=1}^r
A_j^{ab}\,\gamma_b(Mt-j).
\]

Now stack the state components into one curve
\[
\Stack(\Gamma)(t)
=
\begin{pmatrix}
\gamma_1(t)\\
\vdots\\
\gamma_r(t)
\end{pmatrix}
\in\R^{pr}.
\]
For each \(j\in\Z\), let \(\mathcal A_j\in\R^{pr\times pr}\) be the block matrix whose \((a,b)\)-block is
\(A_j^{ab}\), and let
\[
\mathcal B(t)=
\begin{pmatrix}
B_1(t)\\
\vdots\\
B_r(t)
\end{pmatrix}.
\]
These data define the ordinary affine refinement operator
\[
(\mathcal W\widetilde\Gamma)(t)
=
\sum_{j\in\Z}\mathcal A_j\,\widetilde\Gamma(Mt-j)+\mathcal B(t),
\qquad
\widetilde\Gamma:\R\to\R^{pr}.
\]
Its homogeneous part is
\[
(\mathcal V\widetilde\Gamma)(t)
=
\sum_{j\in\Z}\mathcal A_j\,\widetilde\Gamma(Mt-j).
\]

\begin{theorem}[Finite-state embedding]
\label{thm:section5-stateful-embedding}
For every state family \(\Gamma\),
\[
\Stack(\mathfrak W\Gamma)=\mathcal W(\Stack(\Gamma)).
\]
Consequently,
\[
\Stack(\mathfrak W^n\Gamma)=\mathcal W^n(\Stack(\Gamma)),
\qquad
\Stack(\mathfrak V^n\Gamma)=\mathcal V^n(\Stack(\Gamma)).
\]
\end{theorem}

\begin{proof}
The \(a\)-th block of \(\mathcal W(\Stack(\Gamma))(t)\) is
\[
\sum_{j\in\Z}\sum_{b=1}^r
A_j^{ab}\,\gamma_b(Mt-j)+B_a(t),
\]
which is exactly \((\mathfrak W\Gamma)_a(t)\).  This proves the first identity; the iterate formulas
follow immediately by induction.
\end{proof}

\begin{remark}[Deterministic state transitions]
A common special case is the deterministic one: for each state \(a\) and branch \(j\), exactly one
successor state \(\sigma(a,j)\) is active.  Then
\[
(\mathfrak W\Gamma)_a(t)
=
\sum_{j\in\Z}
C_{a,j}\,\gamma_{\sigma(a,j)}(Mt-j)+B_a(t),
\]
for suitable geometric matrices \(C_{a,j}\in\R^{p\times p}\).  Hilbert-type systems usually arise in
this form.
\end{remark}

The embedding theorem reduces the finite-state case directly to the stage-dependent forcing theorem
from \S\ref{ss:stage-dependent}, so the realizability consequence is immediate.
Assume that there exists \(L\ge1\) such that every state component \(\gamma_a\) and every forcing term
\(B_a\) is compactly supported and CPwL in \([0,L]\), and that the finite-state refinement preserves
this support window statewise. Then the stacked operator \(\mathcal W\) satisfies the standing
hypotheses of \Cref{thm:black-box-stage-dependent}, now with target dimension \(pr\).

\begin{corollary}[Finite-state realizability]
\label{cor:section5-network-consequence}
Under the above assumptions, if the initial state curves \(\gamma_a\) have at most \(m_\gamma\)
breakpoints each and the forcing terms \(B_a\) have at most \(m_B\) breakpoints each, then there
exist constants \(C_0,C_1>0\), depending only on \(M\), \(p\), \(r\), \(L\), \(m_\gamma\), and
\(m_B\), such that
\[
\Stack(\mathfrak W^n\Gamma)\in \Ups_{C_0,C_1n^2}(\ReLU;1,pr),
\qquad n\ge1.
\]
In particular, each state component \((\mathfrak W^n\Gamma)_a\) belongs to
\(\Ups_{C_0',C_1'n^2}(\ReLU;1,p)\), and the same conclusion holds for any fixed linear readout of
the stacked state.

If, in addition, all forcing terms vanish identically, so that the system is homogeneous, then the
same argument applied to the homogeneous theorem yields the stronger linear-depth bound
\(O(n)\).
\end{corollary}

\begin{proof}
Apply \Cref{thm:black-box-stage-dependent} to the embedded operator \(\mathcal W\), and then
compose with fixed linear projections or other fixed linear readouts as needed. The homogeneous
improvement is obtained in the special case \(\mathcal B\equiv0\) by applying
\Cref{thm:intro-homogeneous} instead. This proves the claim.
\end{proof}

\begin{example}[Gosper curve \cite{mandelbrot}, \Cref{fig:example}]
\label{ex:gosper-finite-state}
The Gosper curve is naturally encoded as a homogeneous finite-state refinement with two states
\(S=\{A,B\}\) and \(M=7\).
Set
\[
\phi:=\arctan(\sqrt3/5),
\]
and define the branch-angle patterns
\[
\theta_A=\Bigl(0,-\frac\pi3,-\pi,-\frac{2\pi}3,0,0,\frac\pi3\Bigr),
\qquad
\theta_B=\Bigl(\frac\pi3,0,0,-\frac{2\pi}3,-\pi,-\frac\pi3,0\Bigr).
\]
For \(j=0,\dots,6\), let
\[
C_{A,j}:=\frac1{\sqrt7}R_{\phi+\theta_{A,j}},
\qquad
C_{B,j}:=\frac1{\sqrt7}R_{\phi+\theta_{B,j}}.
\]
Let the deterministic state transitions be
\[
\sigma(A,\cdot)=(A,B,B,A,A,A,B),
\qquad
\sigma(B,\cdot)=(A,B,B,B,A,A,B),
\]
and take the endpoint-extended unit segment as the stage-zero state:
\[
\gamma_{A,0}(t)=\gamma_{B,0}(t)=
\begin{cases}
0, & t\le 0,\\
(t,0), & 0\le t\le 1,\\
(1,0), & t\ge 1.
\end{cases}
\]
Then the Gosper recursion is
\[
(\mathfrak V\Gamma)_a(t)
=
\sum_{j=0}^{6} C_{a,j}\,\gamma_{\sigma(a,j)}(7t-j),
\qquad a\in\{A,B\},
\]
for the state vector \(\Gamma=(\gamma_A,\gamma_B)\).

The endpoint-extended class is preserved statewise, since with \(e=(1,0)\) one has
\[
\sum_{j=0}^{6} C_{A,j}e=e,
\qquad
\sum_{j=0}^{6} C_{B,j}e=e.
\]
Thus one may subtract the straight unit-segment anchor profile in each state and pass to compactly
supported defects. The resulting defect recursion is an affine finite-state system, with forcing term
given by the anchor mismatch. Hence, by \Cref{cor:section5-network-consequence}, the stage-\(n\)
geometric Gosper iterates admit exact fixed-width ReLU realizations with quadratic depth
\(O(n^2)\).
\end{example}

\subsection{Homogeneous polygonal generators}
\label{subsec:section6-homogeneous-polygonal}

Let us turn to geometric examples of the preceding corollary framework, beginning with a
forcing-free class of homogeneous polygonal generators. Here one prescribes a polygonal chain and
chooses linear maps so that the refined curve runs through its vertices. This already captures
Koch-, dragon-, and Hilbert-type homogeneous prototypes, and illustrates both the geometric
flexibility and the structural constraints in the choice of the maps \(A_j\).

We now describe a natural class of recursive polygonal rules that fit the homogeneous theory with no
forcing term.
We start from a polygonal chain
\[
P_0,P_1,\dots,P_M\in\R^p,
\]
and pick the endpoint-extended class of curves
\begin{equation}
\label{eq:section6-endpoint-extended-class}
\gamma(t)=P_0 \quad (t\le 0),
\qquad
\gamma(t)=P_M \quad (t\ge 1).
\end{equation}
Without loss of generality, we may assume 
\(P_0 = 0\) and \(P_M=e:=e_1\), where \(e_1\) is the first standard unit vector in \(\R^p\).
We seek matrices \(A_0,\dots,A_{M-1}\in\R^{p\times p}\) such that the homogeneous refinement
\begin{equation}
\label{eq:section6-homogeneous-generator}
(V\gamma)(t):=\sum_{j=0}^{M-1}A_j\gamma(Mt-j)
\end{equation}
has, on each subinterval
\(I_j:=[\frac{j}{M},\frac{j+1}{M}]\),
the form of a transformed copy of the whole curve running from \(P_j\) to \(P_{j+1}\).

The required compatibility is simply
\begin{equation}
\label{eq:section6-edge-condition}
A_j e=P_{j+1}-P_j,
\qquad j=0,\dots,M-1.
\end{equation}
This constrains only the image of the distinguished endpoint direction \(e\), so the construction is
typically far from unique.
If one further asks each \(A_j\) to be a similarity carrying the unit segment to the oriented edge
\(P_{j+1}-P_j\), then the scale is fixed by \(|P_{j+1}-P_j|\), while residual rotational freedom
remains: in the planar case one has orientation-preserving or orientation-reversing choices, and in
higher dimensions one has an orthogonal freedom on \(e^\perp\).

\begin{proposition}[Homogeneous polygonal generator]
\label{prop:section6-homogeneous-generator}
Let \(P_0=0,P_1,\dots,P_{M-1},P_M=e\in\R^p\), and let \(A_0,\dots,A_{M-1}\in\R^{p\times p}\) satisfy
\eqref{eq:section6-edge-condition}.
Then for every curve \(\gamma\) satisfying \eqref{eq:section6-endpoint-extended-class}, the
homogeneous refinement operator \eqref{eq:section6-homogeneous-generator} obeys
\begin{equation}
\label{eq:section6-piecewise-generator}
(V\gamma)(t)=P_j+A_j\gamma(Mt-j),
\qquad t\in I_j.
\end{equation}
In particular, the refined curve runs through the polygonal vertices
\(P_0,P_1,\dots,P_M\).
Moreover, we have
\[
Se=e,
\qquad\text{with}\qquad
S=\sum_{j=0}^{M-1}A_j .
\]
Hence the anchored endpoint criterion of \Cref{lem:affine-anchor-mismatch-criterion} is satisfied for the endpoint-extended class
\eqref{eq:section6-endpoint-extended-class}: the left endpoint \(0\) is automatic, and the right
endpoint condition is exactly \(Se=e\).
\end{proposition}

\begin{proof}
Fix \(t\in I_j\).
Then \(Mt-j\in[0,1]\), while \(Mt-m\ge1\) for \(m<j\) and \(Mt-m\le0\) for \(m>j\).
Using \eqref{eq:section6-endpoint-extended-class}, we obtain
\[
\gamma(Mt-m)=e \quad (m<j),
\qquad
\gamma(Mt-m)=0 \quad (m>j),
\]
and therefore
\[
(V\gamma)(t)=\sum_{m<j}A_m e + A_j\gamma(Mt-j).
\]
By \eqref{eq:section6-edge-condition}, we have
\[
\sum_{m<j}A_m e
=
\sum_{m<j}(P_{m+1}-P_m)
=
P_j,
\]
which proves \eqref{eq:section6-piecewise-generator}.
Finally, we conclude
\[
Se=\sum_{j=0}^{M-1}A_j e
=
\sum_{j=0}^{M-1}(P_{j+1}-P_j)
=
P_M-P_0
=
e.
\]
The last statement now follows from \Cref{lem:affine-anchor-mismatch-criterion} with \(\Gamma_-=0\) and \(\Gamma_+=e\).
\end{proof}

The identity \(Se=e\) connects the endpoint-extended rule to the compactly supported theory.  If
\(\Gamma\) is an anchor profile with endpoint values \(0\) and \(e\), then the homogeneous anchor
mismatch \(D:=V\Gamma-\Gamma\) satisfies the tail condition in
\Cref{lem:affine-anchor-mismatch-criterion}, hence \(D\) is compactly supported.  Therefore, for
\(\gamma=\Gamma+\eta\) with compactly supported defect \(\eta\), the anchored reduction gives
\[
V^n\gamma=\Gamma+\bar W^{\,n}\eta,
\qquad
\bar W\xi:=V\xi+D.
\]
If \(D=0\), the defect evolves homogeneously and \(\eta\)-iterates have the linear-depth bound of
\Cref{thm:intro-homogeneous}; in general, \Cref{thm:black-box-stage-dependent} gives the quadratic-depth bound.

\begin{example}[Dragon curves \cite{Levy,DK}]
Take \(p=2\), \(M=2\), and
\[
P_0=(0,0),\qquad
P_1=\Big(\frac12,\frac12\Big),\qquad
P_2=(1,0).
\]
One convenient choice is
\[
A_0=\frac{1}{\sqrt2}R_{\pi/4},
\qquad
A_1=\frac{1}{\sqrt2}R_{-\pi/4}.
\]
This gives a standard forcing-free two-piece dragon-type recursion (L\'evy-C).  Replacing one of the two
rotations by its reflected version produces the mirror variant (Heighway); in higher dimensions one obtains
further families from the remaining orthogonal freedom.
\end{example}

\begin{example}[Koch curve \cite{Koch}]
Take \(p=2\), \(M=4\), and
\[
P_0=(0,0),\quad
P_1=\Big(\frac13,0\Big),\quad
P_2=\Big(\frac12,\frac{\sqrt3}{6}\Big),\quad
P_3=\Big(\frac23,0\Big),\quad
P_4=(1,0).
\]
With \(R_\theta\) denoting the rotation matrix through angle $\theta$, choose
\[
A_0=\frac13R_0,\qquad
A_1=\frac13R_{\pi/3},\qquad
A_2=\frac13R_{-\pi/3},\qquad
A_3=\frac13R_0.
\]
Then \(A_j e=P_{j+1}-P_j\), so \Cref{prop:section6-homogeneous-generator} gives the usual
forcing-free homogeneous Koch recursion.
\end{example}

\begin{example}[Hilbert-type recursion \cite{Hilbert}, \Cref{fig:example}]
\label{ex:hilbert-type}
Take \(p=2\), \(M=4\), and
\[
P_0=(0,0),\quad
P_1=\Big(0,\frac12\Big),\quad
P_2=\Big(\frac12,\frac12\Big),\quad
P_3=\Big(1,\frac12\Big),\quad
P_4=(1,0).
\]
Choose
\[
A_0=\frac12
\begin{pmatrix}
0&1\\
1&0
\end{pmatrix},
\qquad
A_1=\frac12
\begin{pmatrix}
1&0\\
0&1
\end{pmatrix},
\qquad
A_2=\frac12
\begin{pmatrix}
1&0\\
0&1
\end{pmatrix},
\qquad
A_3=\frac12
\begin{pmatrix}
0&-1\\
-1&0
\end{pmatrix}.
\]
Thus \(A_0\) is one-half of the reflection through the line \(x=y\), \(A_1\) and \(A_2\) are one-half of
the identity, and \(A_3\) is the \(180^\circ\)-rotated version of \(A_0\).
Since
\[
A_0e=\Big(0,\frac12\Big),\qquad
A_1e=\Big(\frac12,0\Big),\qquad
A_2e=\Big(\frac12,0\Big),\qquad
A_3e=\Big(0,-\frac12\Big),
\]
we indeed have
\[
A_j e=P_{j+1}-P_j,
\qquad j=0,1,2,3.
\]
Therefore \Cref{prop:section6-homogeneous-generator} gives a forcing-free homogeneous four-piece
Hilbert-type recursion.

Unlike the standard Hilbert construction, this version uses no additional connector geometry, so the
resulting curve is self-intersecting and geometrically less natural.  Nevertheless, it is a useful
homogeneous prototype showing that Hilbert-type orientation patterns already fit the forcing-free
framework at the level of the linear copy maps.
\end{example}

\begin{figure}[htpb]
    \centering
    \includegraphics[height=0.45\textwidth]{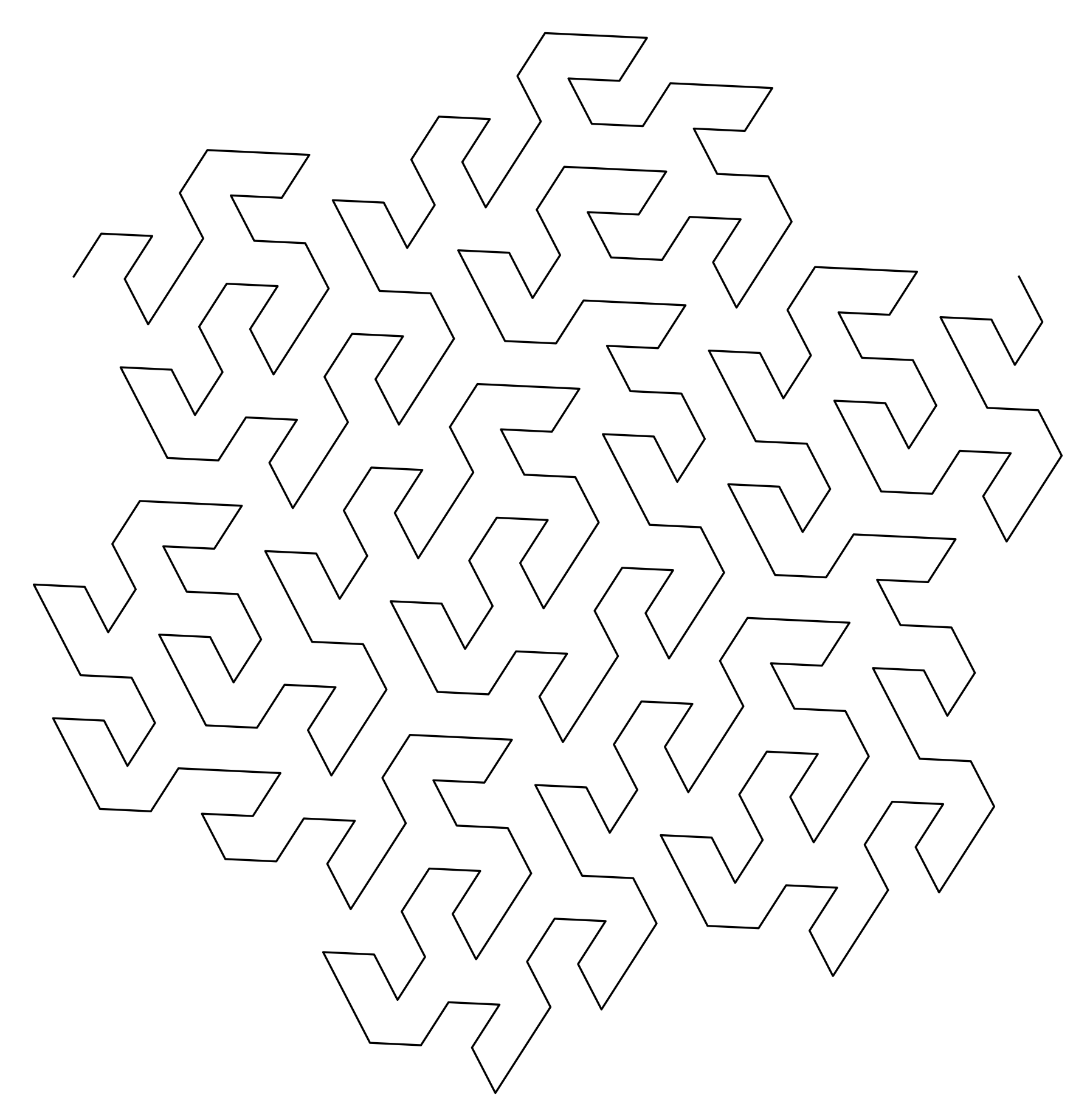}
    \quad
    \includegraphics[height=0.45\textwidth]{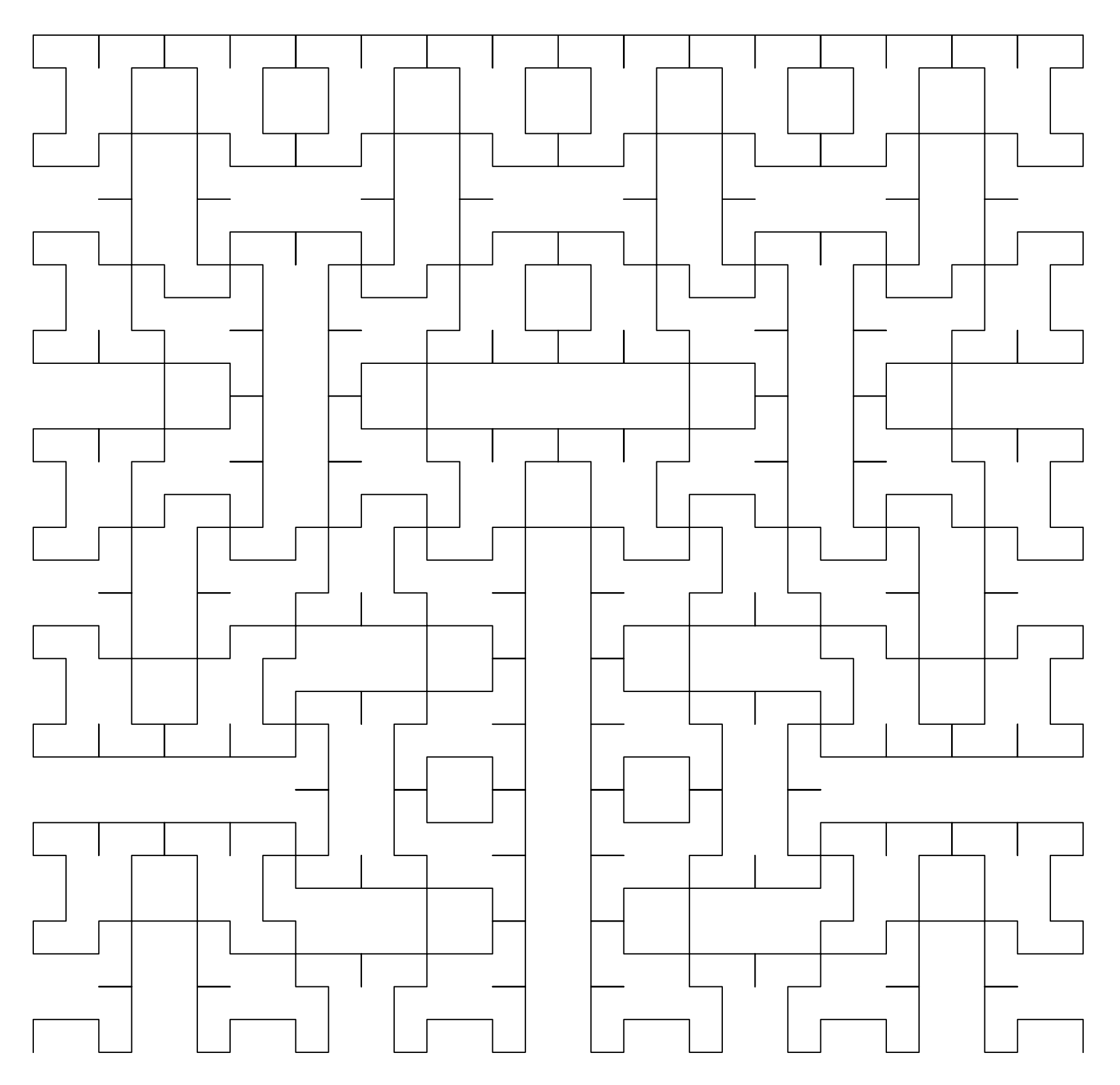}
    \caption{Gosper curve (Example~\ref{ex:gosper-finite-state}) and self-intersecting Hilbert-type curve (Example~\ref{ex:hilbert-type}).}
    \label{fig:example}
\end{figure}

\begin{example}[Hilbert-type recursion in \(\R^p\)]
Let \(p\ge2\), and set \(M=2^p\).
We describe an inductive homogeneous prototype based on the reflected Hilbert symmetry.

For \(p=1\), take
\[
\textstyle P_0^{(1)}=0,\qquad P_1^{(1)}=\frac12,\qquad P_2^{(1)}=1,
\]
and let \(\cube_0^{(1)}=[0,\tfrac12]\), \(\cube_1^{(1)}=[\tfrac12,1]\).

Now assume inductively that, in dimension \(p-1\), we have already constructed a chain
\[
P_0^{(p-1)},P_1^{(p-1)},\dots,P_{2^{p-1}}^{(p-1)}\in[0,1]^{p-1},
\]
from \(0\) to \(e_1\), together with tagged dyadic subcubes
\[
\cube_0^{(p-1)},\cube_1^{(p-1)},\dots,\cube_{2^{p-1}-1}^{(p-1)}.
\]
Write \(n:=2^{p-1}\).
We define the \(p\)-dimensional chain
\[
P_0^{(p)},P_1^{(p)},\dots,P_{2n}^{(p)}\in[0,1]^p
\]
by
\[
P_j^{(p)} = 
\begin{cases}
\bigl(0,P_j^{(p-1)}\bigr) & \text{for}\quad0\le j\le n-1, \\
\bigl(\frac12,P_{n-1}^{(p-1)}\bigr) & \text{for}\quad j = n, \\
\bigl(1,P_{2n-j}^{(p-1)}\bigr) & \text{for}\quad n+1\le j\le 2n.
\end{cases}
\]
Thus the lower copy of the \((p-1)\)-dimensional chain is followed up to its last segment, that
last segment is replaced by a two-step bridge across the first coordinate, and the upper copy is then
traversed in reverse order.

The tagged dyadic subcubes are defined recursively by
\[
\cube_j^{(p)} = 
\begin{cases}
[0,\tfrac12]\times \cube_j^{(p-1)} & \text{for}\quad0\le j\le n-1, \\
[\tfrac12,1]\times \cube_{2n-1-j}^{(p-1)} & \text{for}\quad n\le j\le 2n-1.
\end{cases}
\]
For each \(j=0,\dots,2n-1\), the segment
\([P_j^{(p)},P_{j+1}^{(p)}]\)
is an edge of the dyadic subcube \(\cube_j^{(p)}\), with \(P_j^{(p)}\) as one of its vertices.
Choose a signed permutation matrix \(U_j^{(p)}\in O(p)\) such that
\[
P_j^{(p)}+\frac12\,U_j^{(p)}[0,1]^p = \cube_j^{(p)},
\qquad
\frac12\,U_j^{(p)}e_1=P_{j+1}^{(p)}-P_j^{(p)}.
\]
We then set
\[
A_j^{(p)}:=\frac12\,U_j^{(p)},
\qquad j=0,\dots,2^p-1.
\]
By construction,
\[
A_j^{(p)}e_1=P_{j+1}^{(p)}-P_j^{(p)},
\qquad j=0,\dots,2^p-1,
\]
so \Cref{prop:section6-homogeneous-generator} yields a forcing-free homogeneous
\(2^p\)-piece Hilbert-type recursion in \(\R^p\).

The choice of \(U_j^{(p)}\) is not unique in general for \(p\ge3\), since the image of \(e_1\) fixes
only the tagged edge direction.
To make the choice canonical, write
\[
U_j^{(p-1)}=\bigl[u_j\ \ \widehat U_j\bigr],
\]
where \(u_j\in\R^{p-1}\) is the first column and
\(\widehat U_j\in\R^{(p-1)\times(p-2)}\) collects the remaining columns.
Then define \(U_j^{(p)}\) recursively by
\[
U_j^{(p)}=
\begin{cases}
\begin{pmatrix}
0 & 1 & 0_{1\times(p-2)}\\[0.3em]
u_j & 0 & \widehat U_j
\end{pmatrix},
& 0\le j\le n-2,\\[1.2em]
\begin{pmatrix}
1 & 0_{1\times(p-1)}\\[0.3em]
0_{(p-1)\times1} & U_{n-1}
\end{pmatrix},
& j=n-1,n,\\[1.2em]
\begin{pmatrix}
0 & -1 & 0_{1\times(p-2)}\\[0.3em]
-\,u_{\,2n-1-j} & 0 & \widehat U_{\,2n-1-j}
\end{pmatrix},
& n+1\le j\le 2n-1.
\end{cases}
\]
For \(0\le j\le n-2\), this inserts the new first-coordinate direction \(+e_1\) as the second
column; for the two middle bridge cubes \(j=n-1,n\), the tagged edge itself is \(+e_1\); and for
\(n+1\le j\le 2n-1\), the upper half uses the lower-dimensional matrices in reverse order, with the
tagged edge reversed and the new first-coordinate direction \(-e_1\).
This fixes \(U_j^{(p)}\) uniquely.
\end{example}

\begin{remark}[Affine extension with prescribed ports and connectors]
\label{rem:section6-affine-ported}
The homogeneous construction above extends naturally to an affine one when the consecutive copies do
not meet directly.
Indeed, one may prescribe entry ports
\[
P_0:=0,\qquad P_1,\dots,P_{\ell-1}\in\R^p,
\]
choose matrices \(A_0,\dots,A_{\ell-1}\in\R^{p\times p}\), and define the corresponding exit ports by
\[
Q_j:=P_j+A_j e,
\qquad j=0,\dots,\ell-1.
\]
If one also prescribes CPwL connector curves joining \(Q_j\) to \(P_{j+1}\), then, by subdividing
the parameter interval into alternating copy and connector intervals, one obtains a stationary affine
refinement rule of the form
\[
(W\gamma)(t)=\sum_{j=0}^{\ell-1}A_j\gamma(Mt-2j)+B(t),
\qquad M=2\ell-1,
\]
with a fixed CPwL forcing term \(B\) encoding the port translations and connectors.

The anchored endpoint criterion of \Cref{lem:affine-anchor-mismatch-criterion} is satisfied by
construction.
The left endpoint condition is automatic, since the forcing term is taken to vanish on the left tail.
For the right endpoint, one either imposes
\(Q_{M-1}=e\),
or inserts one final connector joining \(Q_{M-1}\) to \(e\).
\end{remark}

\subsection{Stage-dependent geometric generators}
\label{subsec:section6-stage-dependent}

We now turn to the direct geometric recursions that arise in Morton- and Hilbert-type constructions.
Unlike the stationary affine extension discussed in \Cref{rem:section6-affine-ported}, these rules
typically do \emph{not} preserve fixed endpoints: the stage-\(n\) geometric curve begins and ends at
locations that drift with \(n\) (for instance, at cell centers). Accordingly, one should not impose a
fixed endpoint condition such as \(\gamma_n(0)=0\) and \(\gamma_n(1)=e\) at the geometric level.
Instead, one works with the stage-\(n\) geometric curves themselves and subtracts a stage-dependent
anchor profile, producing compactly supported defects. The corresponding defect recursion then falls
under the stage-dependent forcing theorem \Cref{thm:black-box-stage-dependent}.

Fix an integer \(\ell\ge2\), and set
\(M=2\ell-1\).
Thus the parameter interval is subdivided into \(\ell\) copy intervals and \(\ell-1\) connector intervals:
\[
I_j:=\Big[\frac{2j}{M},\frac{2j+1}{M}\Big],
\qquad j=0,\dots,\ell-1,
\]
and
\[
J_j:=\Big[\frac{2j+1}{M},\frac{2j+2}{M}\Big],
\qquad j=0,\dots,\ell-2.
\]
For each \(j=0,\dots,\ell-1\), fix an affine map
\begin{equation}
\label{eq:section6-visible-affine-copy-map}
F_j(x):=A_jx+u_j,
\qquad
A_j\in\R^{p\times p},\quad u_j\in\R^p.
\end{equation}

Let \(\gamma_n:[0,1]\to\R^p\) denote the stage-\(n\) geometric curve, and write
\begin{equation}
\label{eq:section6-visible-endpoints}
a_n:=\gamma_n(0),
\qquad
b_n:=\gamma_n(1).
\end{equation}
For each stage \(n\), define the entry and exit ports of the \(j\)-th copy by
\begin{equation}
\label{eq:section6-visible-ports}
P_j^{(n)}:=F_j(a_n),
\qquad
Q_j^{(n)}:=F_j(b_n),
\qquad j=0,\dots,\ell-1.
\end{equation}
For \(j=0,\dots,\ell-2\), let
\[
\zeta_j^{(n)}:[0,1]\to\R^p
\]
be a CPwL connector curve satisfying
\begin{equation}
\label{eq:section6-visible-connector-endpoints}
\zeta_j^{(n)}(0)=Q_j^{(n)},
\qquad
\zeta_j^{(n)}(1)=P_{j+1}^{(n)}.
\end{equation}
We then define the stage-\((n+1)\) geometric curve \(\gamma_{n+1}\) by the piecewise rule
\begin{align}
\gamma_{n+1}(t)
&=
F_j(\gamma_n(Mt-2j)),
&& t\in I_j,\quad j=0,\dots,\ell-1,
\label{eq:section6-visible-copy-rule}
\\
\gamma_{n+1}(t)
&=
\zeta_j^{(n)}(Mt-(2j+1)),
&& t\in J_j,\quad j=0,\dots,\ell-2.
\label{eq:section6-visible-connector-rule}
\end{align}

\begin{proposition}[Geometric stage-dependent generator]
\label{prop:section6-visible-generator}
The rule \eqref{eq:section6-visible-copy-rule}--\eqref{eq:section6-visible-connector-rule} has the
following properties.
For each \(j=0,\dots,\ell-1\), on the copy interval \(I_j\) the refined curve is exactly the image
of the whole previous-stage curve under the affine map \(F_j\), while on each connector interval
\(J_j\) it follows the prescribed connector \(\zeta_j^{(n)}\).
In particular, \(\gamma_{n+1}\) runs through the sequence
\[
P_0^{(n)},\ Q_0^{(n)},\ P_1^{(n)},\ Q_1^{(n)},\ \dots,\ P_{\ell-1}^{(n)},\ Q_{\ell-1}^{(n)},
\]
and therefore its endpoints are
\begin{equation}
\label{eq:section6-visible-endpoint-update}
a_{n+1}=P_0^{(n)}=F_0(a_n),
\qquad
b_{n+1}=Q_{\ell-1}^{(n)}=F_{\ell-1}(b_n).
\end{equation}
\end{proposition}

\begin{proof}
The piecewise identities are exactly \eqref{eq:section6-visible-copy-rule} and
\eqref{eq:section6-visible-connector-rule}. Evaluating them at the interval endpoints shows that
\(\gamma_{n+1}\) visits the displayed sequence of ports, and the endpoint identities
\eqref{eq:section6-visible-endpoint-update} are just its first and last terms.
\end{proof}

To bring this into the scope of our realizability theorems, we pass to anchored defects.
For each \(n\ge0\), let \(\Gamma_n:\R\to\R^p\) be a CPwL anchor profile satisfying
\begin{equation}
\label{eq:section6-visible-anchor}
\Gamma_n(t)=a_n \quad (t\le0),
\qquad
\Gamma_n(t)=b_n \quad (t\ge1),
\end{equation}
and define the defect curve
\begin{equation}
\label{eq:section6-visible-defect}
\eta_n:=\gamma_n-\Gamma_n.
\end{equation}
Then \(\eta_n\) is compactly supported in \([0,1]\).

\begin{proposition}[Anchored reduction]
\label{prop:section6-visible-reduction}
Let \(\eta_n\) be defined by \eqref{eq:section6-visible-defect}. Then
\begin{equation}
\label{eq:section6-visible-affine-recursion}
\eta_{n+1}(t)
=
\sum_{j=0}^{\ell-1}A_j\,\eta_n(Mt-2j)+B_n(t),
\qquad t\in\R,
\end{equation}
where the forcing term \(B_n:\R\to\R^p\) is compactly supported and CPwL, and is given on
\([0,1]\) by
\begin{align}
B_n(t)
&=
F_j\bigl(\Gamma_n(Mt-2j)\bigr)-\Gamma_{n+1}(t),
&& t\in I_j,\quad j=0,\dots,\ell-1,
\label{eq:section6-visible-forcing-copy}
\\
B_n(t)
&=
\zeta_j^{(n)}(Mt-(2j+1))-\Gamma_{n+1}(t),
&& t\in J_j,\quad j=0,\dots,\ell-2,
\label{eq:section6-visible-forcing-connector}
\end{align}
and by \(B_n(t)=0\) for \(t\notin[0,1]\).
\end{proposition}

\begin{proof}
On each copy interval \(I_j\), substitute \(\gamma_n=\eta_n+\Gamma_n\) into
\eqref{eq:section6-visible-copy-rule} and subtract \(\Gamma_{n+1}\). On each connector interval
\(J_j\), subtract \(\Gamma_{n+1}\) from \eqref{eq:section6-visible-connector-rule}. This gives
\eqref{eq:section6-visible-affine-recursion} together with the formulas
\eqref{eq:section6-visible-forcing-copy}--\eqref{eq:section6-visible-forcing-connector}.
The compact support and CPwL character of \(B_n\) are immediate from these piecewise formulas.
\end{proof}

The stage-dependent forcing theorem applies once the forcing terms \(B_n\) lie in a fixed finite-
dimensional CPwL family. For the straight-anchor and straight-connector choices used below, this
reduces to a condition on the endpoint sequences alone.

\begin{proposition}[Endpoint templates imply forcing templates]
\label{prop:section6-visible-fixed-span}
Assume that there exist scalar sequences \(\lambda_{n,\alpha}\in\R\) (\(n\ge0\), \(\alpha=0,\dots,N\)),
with
\[
\lambda_{n,0}\equiv1,
\]
and vectors \(a^{(\alpha)},b^{(\alpha)}\in\R^p\) such that
\begin{equation}
\label{eq:section6-endpoint-template-expansion}
a_n=\sum_{\alpha=0}^N \lambda_{n,\alpha}a^{(\alpha)},
\qquad
b_n=\sum_{\alpha=0}^N \lambda_{n,\alpha}b^{(\alpha)}.
\end{equation}
Choose the anchor profiles by the straight-anchor rule
\[
\Gamma_n(t):=a_n+\theta(t)\bigl(b_n-a_n\bigr),
\qquad
\theta(t):=
\begin{cases}
0, & t\le0,\\
t, & 0\le t\le1,\\
1, & t\ge1,
\end{cases}
\]
and, for each \(j=0,\dots,\ell-2\), choose the connector \(\zeta_j^{(n)}\) to be the straight segment
from \(Q_j^{(n)}\) to \(P_{j+1}^{(n)}\):
\[
\zeta_j^{(n)}(s):=(1-s)Q_j^{(n)}+sP_{j+1}^{(n)},
\qquad 0\le s\le1.
\]

Then there exist compactly supported CPwL curves
\[
B^{(0)},\dots,B^{(N)}:\R\to\R^p
\]
such that
\begin{equation}
\label{eq:section6-visible-forcing-span}
B_n=\sum_{\alpha=0}^N \lambda_{n,\alpha}B^{(\alpha)},
\qquad n\ge0.
\end{equation}
Consequently, the defect recursion \eqref{eq:section6-visible-affine-recursion} falls within the
scope of \Cref{thm:black-box-stage-dependent}.
\end{proposition}

\begin{proof}
Define the anchor templates
\[
\Gamma^{(\alpha)}(t):=a^{(\alpha)}+\theta(t)\bigl(b^{(\alpha)}-a^{(\alpha)}\bigr),
\qquad \alpha=0,\dots,N.
\]
Then \eqref{eq:section6-endpoint-template-expansion} gives
\[
\Gamma_n=\sum_{\alpha=0}^N \lambda_{n,\alpha}\Gamma^{(\alpha)}.
\]

Next, define the port templates by
\[
P_j^{(0)}:=u_j+A_ja^{(0)},
\qquad
Q_j^{(0)}:=u_j+A_jb^{(0)},
\]
and, for \(\alpha\ge1\),
\[
P_j^{(\alpha)}:=A_ja^{(\alpha)},
\qquad
Q_j^{(\alpha)}:=A_jb^{(\alpha)}.
\]
Since \(\lambda_{n,0}\equiv1\), the endpoint expansions imply
\[
P_j^{(n)}=\sum_{\alpha=0}^N \lambda_{n,\alpha}P_j^{(\alpha)},
\qquad
Q_j^{(n)}=\sum_{\alpha=0}^N \lambda_{n,\alpha}Q_j^{(\alpha)}.
\]
Hence each straight connector \(\zeta_j^{(n)}\) admits the same expansion, and substituting these
anchor and connector expansions into
\eqref{eq:section6-visible-forcing-copy}--\eqref{eq:section6-visible-forcing-connector}
gives \eqref{eq:section6-visible-forcing-span}.
The final claim is then immediate from \Cref{thm:black-box-stage-dependent}.
\end{proof}

\begin{example}[Hilbert curves \cite{Hilbert}]
Take \(p=2\), \(\ell=4\), and hence \(M=7\).
Let
\[
A_0=\frac12
\begin{pmatrix}
0&1\\
1&0
\end{pmatrix},
\qquad
A_1=\frac12 I_2,
\qquad
A_2=\frac12 I_2,
\qquad
A_3=\frac12
\begin{pmatrix}
0&-1\\
-1&0
\end{pmatrix},
\]
and
\[
u_0=(0,0),\qquad
u_1=(0,\tfrac12),\qquad
u_2=(\tfrac12,\tfrac12),\qquad
u_3=(1,\tfrac12).
\]
Thus \(F_j(x)=A_jx+u_j\), \(j=0,1,2,3\), places the four copies in the dyadic subsquares in the
usual Hilbert order.

The stage-\(n\) geometric curve has endpoints
\[
a_n=\Bigl(\frac{\lambda_n}{2},\frac{\lambda_n}{2}\Bigr),
\qquad
b_n=\Bigl(1-\frac{\lambda_n}{2},\frac{\lambda_n}{2}\Bigr),
\qquad
\lambda_n:=2^{-n}.
\]
Equivalently,
\[
a_n=a^{(0)}+\lambda_n a^{(1)},
\qquad
b_n=b^{(0)}+\lambda_n b^{(1)},
\]
with
\[
a^{(0)}=(0,0),\qquad a^{(1)}=\Bigl(\frac12,\frac12\Bigr),
\qquad
b^{(0)}=(1,0),\qquad b^{(1)}=\Bigl(-\frac12,\frac12\Bigr).
\]
Choose straight anchors and straight connectors. Then
\Cref{prop:section6-visible-fixed-span} gives
\[
B_n=B^{(0)}+\lambda_n B^{(1)}.
\]
Hence the direct geometric Hilbert recursion falls within the scope of
\Cref{thm:black-box-stage-dependent}.
\end{example}

\begin{example}[Morton curves \cite{Lebesgue,Morton}]
Let \(p\ge1\), set
\[
\ell=2^p,
\qquad
M=2\ell-1=2^{p+1}-1,
\]
and enumerate the vertices of \(\{0,1\}^p\) in increasing binary order by
\[
\varepsilon^{(0)},\varepsilon^{(1)},\dots,\varepsilon^{(\ell-1)}\in\{0,1\}^p.
\]
For each \(r=0,\dots,\ell-1\), define
\[
A_r=\frac12 I_p,
\qquad
u_r=\frac12\,\varepsilon^{(r)},
\qquad
F_r(x)=A_r x+u_r.
\]
Thus \(F_r\) places a copy of the unit cube into the dyadic subcube
\(\frac12\varepsilon^{(r)}+[0,\frac12]^p\),
and the \(\ell\) copies are traversed in Morton order.

The stage-\(n\) geometric curve has endpoints
\[
a_n=\frac{\lambda_n}{2}\mathbf1,
\qquad
b_n=\Bigl(1-\frac{\lambda_n}{2}\Bigr)\mathbf1,
\qquad
\lambda_n:=2^{-n},
\]
where \(\mathbf1=(1,\dots,1)\in\R^p\).
Equivalently,
\[
a_n=a^{(0)}+\lambda_n a^{(1)},
\qquad
b_n=b^{(0)}+\lambda_n b^{(1)},
\]
with
\[
a^{(0)}=0,\qquad a^{(1)}=\frac12\mathbf1,
\qquad
b^{(0)}=\mathbf1,\qquad b^{(1)}=-\frac12\mathbf1.
\]
Choose straight anchors and straight connectors. Then
\Cref{prop:section6-visible-fixed-span} gives
\[
B_n=B^{(0)}+\lambda_n B^{(1)}.
\]
Hence the direct geometric Morton recursion in \(\R^p\) falls within the scope of
\Cref{thm:black-box-stage-dependent}.
\end{example}

\section{Conclusions}
\label{sec:conclusions}

We have developed a homogeneous \(M\)-ary vector-valued extension of the scalar binary refinable-
function construction of \cite{source} for compactly supported CPwL curves.  The main result is
that the homogeneous iterates \(V^n\gamma\) admit exact ReLU realizations of fixed width and depth
\(O(n)\).  The central new ingredient is an exact loop controller for the residual dynamics: instead
of iterating scalar surrogate residuals, the construction transports the residual orbit exactly on a
polygonal loop and recovers the scalar factor and selector matrices from that loop by CPwL
readouts.

The loop-controller viewpoint separates exact forward transport from seam handling.  The residual
state is exact at every stage, while the remaining ambiguity at the loop seam is confined to the
terminal readout/selector stage and absorbed by the support separation of the special hat.  This
gives a cleaner and more geometric proof architecture than the scalar-surrogate bookkeeping, while
retaining the fixed-width, linear-depth complexity.

Beyond the main homogeneous theorem, the extended version records several finite-sum reductions
and geometric applications.  Affine and stage-dependent forcing rules are expanded into homogeneous
pieces and treated term by term, giving exact fixed-width realizations with quadratic depth
\(O(n^2)\).  Open curves are handled by subtracting an anchor profile; the compactly supported
defect then satisfies an affine recursion whose forcing term is the anchor mismatch.  Thus, except
when this mismatch vanishes, anchored open curves naturally fall under the affine quadratic-depth
corollary rather than the homogeneous linear-depth theorem.

The geometric examples show how recursive polygonal rules enter the same language.  Forcing-free
homogeneous generators, such as dragon- and Koch-type rules, are encoded by matrices satisfying
\(A_j e=P_{j+1}-P_j\).  More elaborate anchored, affine, or stage-dependent constructions, including
Gosper-, Morton-, and Hilbert-type variants, are covered by the corresponding finite-sum reductions.
In this way, a broad class of recursive curve constructions can be organized around the homogeneous
loop-controller theorem.

Natural next steps include a genuinely recursive affine construction with linear-depth complexity,
extensions to multidimensional parameter domains, and sharper bounds on weights and biases.  For
limiting refinable objects, the present exact finite-iterate theorem can be combined with the standard
cascade convergence arguments of \cite{source} to recover the corresponding asymptotic
approximation statements.

\section*{Acknowledgments}

This work was supported by NSERC Discovery Grants Program.

\end{document}